\documentclass[11pt]{amsart}

\usepackage{graphicx}
\usepackage{amsmath,amsthm,amscd,amssymb}
\usepackage[maxbibnames=99]{biblatex}
\usepackage[cal=boondoxo]{mathalfa}
\usepackage{amssymb}
\usepackage{mathabx}  
\usepackage{bbm}
\usepackage{mathtools} 
\mathtoolsset{showonlyrefs}

\usepackage{verbatim}
\usepackage{enumerate, verbatim}
\theoremstyle{thmstyleone}%
\newtheorem{theorem}{Theorem}
\newtheorem{proposition}[theorem]{Proposition}%

\theoremstyle{thmstyletwo}%
\newtheorem{remark}{Remark}%

\theoremstyle{thmstylethree}%
\newtheorem{definition}{Definition}%

\usepackage{color}
\usepackage{enumitem} 

\numberwithin{equation}{section}
\allowdisplaybreaks

\newcommand{\del}{\partial}
\newcommand{\al}{\alpha}

\newcommand{\la}{\lambda}

\newcommand{\eps}{\varepsilon}

\newcommand{\R}{\mathbb{R}}
\newcommand{\Proj}{\mathsf{\Pi}}

\newcommand{\Density}{\rho}
\newtheorem{lemma}[theorem]{Lemma}

\newtheorem{notation}[theorem]{Notation}
\newtheorem{assumption}[theorem]{Assumption}

\begin{document}

\title{Breakdown of smooth solutions to the M\"uller-Israel-Stewart equations of
relativistic viscous fluids}
\author{Marcelo M. Disconzi, Vu Hoang, Maria Radosz}
\maketitle

\begin{abstract}
We consider equations of M\"uller-Israel-Stewart type describing a relativistic viscous fluid  with bulk viscosity in four-dimensional Minkowski space.
We show that there exists a class of smooth initial data that are 
localized perturbations of constant states 
for which the corresponding unique solutions to the Cauchy problem break down in finite time.
Specifically, we prove that in finite time such solutions develop a singularity or
become unphysical in a sense that we make precise. We also show that in general Riemann invariants do not exist in 1+1 dimensions
for physically relevant equations of state and viscosity coefficients. Finally, we present a more general version of a result by Y. Guo and A.S. Tahvildar-Zadeh: 
we prove large-data singularity formation results for perfect fluids under very general assumptions on the equation of state, allowing any value for the fluid sound speed strictly less than the speed of light.
\end{abstract}

\newcommand{\Addresses}{
\bigskip
  \footnotesize
\noindent M. M. Disconzi, \textsc{Department of Mathematics, Vanderbilt University, 1326 Stevenson Center Ln, Nashville, TN
37212 (USA)}\\
\noindent  V.~Hoang, M. Radosz,  \textsc{Department of Mathematics, University of Texas at San Antonio,
    San Antonio, Texas 78249 (USA)}\par\nopagebreak
 \vspace{0.3cm}
   \noindent \textit{E-mail address}, M. M. Disconzi \texttt{marcelo.disconzi@vanderbilt.edu}\\
  \noindent \textit{E-mail address}, V.~Hoang: \texttt{duynguyenvu.hoang@utsa.edu}\\
  \textit{E-mail address}, M.~Radosz: \texttt{maria\_radosz@hotmail.com}
}











\section{Introduction}
Relativistic hydrodynamics describes the motion of fluids in regimes where relativistic effects are important. This includes flow velocities close to the speed of light and fluids interacting with strong gravitational fields, such as in relativistic plasma produced in heavy-ion collisions or the fluid description of neutron star mergers and black hole accretion disks.  Applied to a wide range of physical phenomena on the largest and smallest length scales, it serves as an essential tool in high-energy nuclear physics, astrophysics and cosmology \cite{Heinz:2013th,Baier:2007ix,RezzollaZanottiBookRelHydro,WeinbergCosmology}. 
The study of relativistic fluid equations started with Einstein \cite{Einstein:1914bx} and Schwarzschild \cite{Schwarzschild1916}, considering perfect fluids, in which case one obtains the well-known
relativistic Euler equations. The mathematical study of 
relativistic perfect fluids goes back to the works of Choquet-Bruhat \cite{Choquet-BruhatFluidsExistence} and Lichnerowicz \cite{Lichnerowicz_MHD_book}, and it is nowadays a very active field of research.
A review of the literature on the mathematical treatment of the relativistic Euler equations
is beyond our scope; we refer the reader to the 
monographs \cite{ChoquetBruhatGRBook,ChristodoulouShocks,ChristodoulouShockDevelopment,SpeckBook}
and references therein.

While the relativistic Euler equations provide a good model to study many physical phenomena and are a rich source of mathematical problems, a physically more complete description of a fluid includes dissipative processes such as viscosity, diffusion and heat conduction. In fact, a thorough understanding of relativistic viscosity is highly relevant to applications in the fundamental physics of the quark-gluon plasma produced in experiments at the Relativistic Heavy Ion Collider and the Large Hadron Collider. For the quark-gluon plasma, it is well-attested that theoretical predictions do not match experimental data if viscosity is not taken into account, and, therefore, the description of the quark-gluon plasma in terms of perfect fluids is regarded as inadequate and was essentially abandoned \cite{Heinz:2013th,Romatschke:2017ejr}. Relativistic viscous fluids are also poised to play a key role in understanding neutron star mergers and to provide information about the properties of high density, degenerate matter. Recent state-of-the art numerical
simulations strongly suggest that the gravitational-wave signal 
of neutron star mergers is likely affected by viscous effects \cite{Alford2018,Shibata:2017jyf,Shibata:2017xht}, a result corroborated by recent studies of the underlying microscopic physics of such mergers \cite{Most:2022yhe}. 
Lastly, relativistic viscous fluids might also be relevant in cosmology due to a variety of dissipative 
processes such as the decoupling of neutrinos and radiation from matter during the early universe 
\cite{Martens, Marcelo2017,Brevik:2020psp}, even though current observations significantly constraint
the inclusion of dissipative effects into cosmological models \cite{LiBarrow,Brevik:2014cxa}.

On the mathematical side, viscous relativistic fluid equations provide a variety of difficult and
interesting problems,
and not much is known about their mathematical properties. 
Although different models have been proposed to describe the dynamics of relativistic
fluids with viscosity, they are all far more complex than the relativistic Euler equations.
Consequently, many basic questions, including the local well-posedness of the Cauchy problem,
remain largely open. 
On the one hand, models need to be compatible with the large
amount of experimental data on properties of relativistic viscous fluids,
at least when it comes to the quark-gluon-plasma 
\cite{Petersen-2017,STAR:2017ckg,2015LongRange}. Experimental data is much 
more scarce for neutron star mergers \cite{Most:2018eaw}, but this might soon
change now that one can use gravitational waves to probe previously inaccessible 
properties of neutron stars \cite{Most:2018eaw,Most:2019onn,Hammond:2022uua}. On the other hand, theories of relativistic viscous
fluids need to respects basic physical principles such as causality and 
linear stability. 
More precisely, linear stability of constant states
describing thermodynamic equilibrium, henceforth referred to simply as
stability \cite{RezzollaZanottiBookRelHydro}.

Causality, which roughly states that no information propagates
faster than the speed of light, is an essential property of any relativistic theory
(see Section \ref{S:Causality} for a precise definition), whereas
stability is expected to hold whenever dissipative effects are present
\cite{RezzollaZanottiBookRelHydro}. 
While causality and stability are a non-issue for most relativistic matter models,
it turns out that it is very difficult to construct phenomenologically relevant models 
of relativistic viscous fluids that 
respect causality and stability. We refer the reader to
\cite{Baier:2007ix,
Lehner:2017yes,
Denicol:2012cn,
Strickland:2014eua,
Hoult:2020eho,
Kovtun:2019hdm,
DisconziBemficaNoronhaConformal,
DisconziBemficaNoronhaBarotropic,
DisconziBemficaHoangNoronhaRadoszNonlinearConstraints,
Bemfica:2020zjp} 
and references therein for more details
(see also \cite{RezzollaZanottiBookRelHydro,Romatschke:2017ejr} for an overview).

In this work, we consider the theory of relativistic viscous fluids
introduced in the works of Israel, Stewart and M\"uller \cite{MIS-1, MIS-2, MIS-3, MIS-4, MIS-5, hiscock_salmonson_1991}, and whose modern versions have been developed
in \cite{Baier:2007ix,Denicol:2012cn}, commonly referred to as the M\"uller-Israel-Stewart 
equations.
We remark that the theories introduced in
\cite{MIS-1, MIS-2, MIS-3, MIS-4, MIS-5, hiscock_salmonson_1991,Baier:2007ix,Denicol:2012cn}
are different from each other, but they share many similarities. It is convenient, therefore, for this introductory discussion, to lump them together as ``the" M\"uller-Israel-Stewart theory. This will not cause confusion because after our introductory discussion we will consider a specific choice of equations of motion, see Definition \ref{D:MIS_equations}.
More precisely, we consider the version of these
equations where the only viscous effects are due to bulk viscosity, as presented in 
\cite{BemficaDisconziNoronha_IS_bulk}. Our choice is motivated by several reasons.
First, the M\"uller-Israel-Stewart equations have been very successful in the construction of phenomenological
models of the quark-gluon plasma. 
With the help of sophisticated numerical simulations (see, e.g., \cite{Ryu:2017qzn}), these models
are able to reproduce, to a great degree of accuracy, many of the experimentally
observed properties of the quark-gluon plasma \cite{Romatschke:2017ejr}.
Because of this, they are currently the most
used equations in the study of viscous effects in relativistic fluids \cite{Romatschke:2017ejr}. 
Second, the M\"uller-Israel-Stewart equations have been shown to be
stable \cite{Hiscock_Lindblom_stability_1983,Olson:1989ey} and, under natural physical
assumptions, to be causal \cite{DisconziBemficaHoangNoronhaRadoszNonlinearConstraints}. Finally,
for the particular case when only bulk viscosity is present among dissipative effects, 
the Cauchy problem for the M\"uller-Israel-Stewart equations has been shown to be 
locally well-posed in Sobolev and smooth spaces under some 
natural assumptions \cite{BemficaDisconziNoronha_IS_bulk}.

A natural question that follows from the considerations in the previous paragraph is
whether, when only bulk viscosity is present,
a solution to the Cauchy problem persists or breaks down in finite time. 
Although the study of breakdown of solutions has a long history in mathematics and physics
(see, e.g., the monographs
\cite{SpeckBook,
BressanBookConservationLaws-2000,
RezzollaZanottiBookRelHydro, 
Courant_and_Hilbert_book_2}), and is an active field of research in the case of relativistic Euler equations \cite{ChristodoulouShocks,ChristodoulouShockDevelopment,DisconziSpeckRelEulerNull},
to the best of our knowledge, this important question
has not been investigated for the M\"uller-Israel-Stewart equations.
To be more precise,
shock singularities for the M\"uller-Israel-Stewart equations were studied in \cite{GerochLindblomCausal,HiscockShocksMIS}. 
These works, however, \emph{assume} that solutions break down due the formation of a shock
singularity, and then proceed to study whether solutions can be continued in a weak sense past the shock.
No argument is given in \cite{GerochLindblomCausal,HiscockShocksMIS}
to show that solutions in fact break down in finite time.
For the study of shocks in a different theory of relativistic viscous fluids, see \cite{TempleViscous}.
In the present manuscript, we provide an answer to this question. More precisely, we establish the following:

\begin{quotation}
\textbf{Main result (see Theorem \ref{T:Main_theorem_blowup_simpler} for a precise statement):}
Consider the M\"uller-Israel-Stewart equations and assume that 
the only dissipative contribution is given by bulk viscosity; see Definition \ref{D:MIS_equations}
for the equations of motion. We show that there exists a class of smooth initial data for which 
the corresponding solutions to the Cauchy problem break down in finite time.
More precisely, there exists a $T_0>0$ such that solutions cannot be continued as $C^1$ solutions
all the way up to $t=T_0$ or become unphysical for $t=T_0$. ``Unphysical" here  {means that solutions
become acausal. Causality is understood in the usual sense
of relativity theory, see Definition \ref{D:Causality}, whereas
the notion of physical solutions is introduced in Definition \ref{D:physical_solution_of_MIS}.}
The set of initial data that we construct consists of perturbations of constant states in the sense
that they have constant density, constant baryon number, constant velocity, and zero bulk viscosity
outside a ball of fixed radius.
\end{quotation}

The perturbations we construct are large perturbations of constant states localized in small regions.
It is a  natural question to ask what happens in the case of small perturbations. This is a much harder question. Even in the case of non-relativistic ideal fluids, only recently this question has started to be understood outside
symmetry \cite{SpeckLukShocks2dEuler,LukSpeckShock3DEuler-2021-arxiv}, and we are not aware of any
similar result for the relativistic Euler equations, much less the M\"uller-Israel-Stewart equations. 
Under symmetry assumptions, typically much more can be said about the formation of singularities or global existence for many classes
of fluids equations (see \cite{SpecketalOverviewShocks,SpeckSummaryShocks} for an overview
and \cite{Rendall-Stahl-2008} for the case of the Einstein-Euler system), but these results do not directly
apply to the M\"uller-Israel-Stewart equations.  In particular,
there seems to be no global results for small data for the 
M\"uller-Israel-Stewart system (except, of course, for the trivial case of constant
equilibrium states; see Section \ref{S:MIS} for the definition).


In Section \ref{S:Setting} we provide the set-up and definitions needed to state our results.
Within that Section, we briefly review the basic formalism of theories of relativistic viscous
fluids in Section \ref{S:Overview} and introduce the M\"uller-Israel-Stewart equations
in Section \ref{S:MIS}. The concepts of causality and the weak energy condition are
reviewed in Sections \ref{S:Causality} and \ref{S:Dominant}, respectively.
The concept of admissible solutions, which is important for the statement of our results, is given in 
Section \ref{S:Admissible}. The main results about viscous and perfect fluids are stated in Section \ref{S:Results}, and their
physical and mathematical significance is discussed in Section \ref{S:Significance}. The proof
of the main result is given in Section \ref{S:Proof}. Riemann invariants for the system in $1+1$
dimensions are discussed in Section \ref{S:Riemann}. The connection between the study of Riemann
invariants in $1+1$ dimensions and our main result is expalined in Section \ref{S:Significance}.

\section{Setting and statement of the results}
\label{S:Setting}
In this Section we provide the set-up and definitions needed
to precisely state our main result, and make further comments about the physical and mathematical 
significance of our results.

\subsection{Overview: from perfect to viscous fluids}
\label{S:Overview}
On four-dimensional space-time, a perfect fluid is characterized by its energy-momentum tensor
\begin{align}
T_{\alpha\beta} := \Density u_\alpha u_\beta + p\Proj_{\alpha\beta},
\label{E:Energy_momentum_perfect}
\end{align}
where $u$ is the fluid's (four-)velocity,
which is a future-directed, timelike vector field 
satisfying the normalization condition
\begin{align}
g_{\alpha\beta} u^\alpha u^\beta = -1,
\label{E:Normalization_u}
\end{align}
where $g$ is the space-time metric.
$\Proj$ is the projector onto the space orthogonal to $u$, given
by 
\begin{align}
    \Proj_{\alpha\beta} = g_{\alpha\beta} + u_\alpha u_\beta
    \nonumber
\end{align} 
for $u$ satisfying
\eqref{E:Normalization_u},
 $\Density$ is the (energy) density of the fluid and $p$ its pressure. 
We henceforth assume that the fluid is embedded in the
four-dimensional Minkowski space $\mathbb{R}^{1+3}$. Above 
and throughout we adopt the following:

\begin{notation}
Greek indices run from $0$ to $3$, Latin indices vary from $1$ to $3$, and repeated indices are summed over their range.
Expressions like $z_\alpha$, $w_{\alpha\beta}$, etc., denote the components
of a tensor relative to Cartesian coordinates $\{ x^\alpha \}_{\alpha=0}^3$ in $\mathbb{R}^{1+3}$,
where $t := x^0$ denotes a time coordinate. Indices are raised and lowered with the Minkowski metric $g$,
given in Cartesian coordinates by
\begin{align}
    g = \operatorname{diag}(-1, +1, +1, +1).
    \nonumber
\end{align}
$\R^+$ denotes the open interval $(0, \infty)$ and $B_R = \{ x \in \mathbb{R}^3 : \vert x\vert  < R \}$ the open ball with radius $R$ in three-dimensional space. We work in units where the speed of light in vacuum equals to one.
\end{notation}

A basic postulate of relativity is the \emph{conservation of energy and momentum}, expressed by
\begin{align}
\nabla_\alpha T^{\alpha}_\beta = 0, \label{E:Div_T}
\end{align}
where $\nabla$ is the covariant derivative associated with the metric $g$ (so that $\nabla_\alpha=
\partial_\alpha$ for the Minkowski metric). Equations \eqref{E:Energy_momentum_perfect},
\eqref{E:Normalization_u}, and \eqref{E:Div_T} imply the relativistic Euler equations:
\begin{align}
	u^\alpha \nabla_\alpha \Density + (p+\Density) \nabla_\alpha u^\alpha &= 0,
	\nonumber
	\\
    (\Density+p) u^\al \nabla_\al u^\beta + \Proj^{\beta \gamma} \nabla_\gamma p &= 0.
    \nonumber
\end{align}
In order to close the system, one needs an equation of state. In general, it is known that the pressure can depend on various thermodynamic
quantities. Using the laws of thermodynamics, one can assume without loss of generality that
$p$ is determined by at most two thermodynamic scalars \cite{AnileBook}, which here we will take to be 
the density $\Density$ and the particle number $n$, i.e., $p=p(\Density,n)$. The latter is postulated to be conserved
in the sense that
\begin{align}
	\nabla_\alpha (n u^\alpha) = 0.
	\label{E:Conservation_n_initial}
\end{align}
The particle number can be interpreted, up to a dimensional constant, as the rest mass
density of the fluid. Thus, \eqref{E:Conservation_n_initial} can be thought of as a relativistic generalization
of the conservation of mass in non-relativistic physics, see \cite{RezzollaZanottiBookRelHydro} for details.

The development of theories of relativistic fluids with viscosity seek to modify
\eqref{E:Energy_momentum_perfect}. This is natural since one would like to recover the equations
of a relativistic perfect fluid when the viscous contributions vanish. 
The first attempt in this direction was proposed by Eckart \cite{EckartViscous}, followed by a similar proposal by Landau-Lifshitz \cite{LandauLifshitzFluids}. Including bulk viscous effects alters \eqref{E:Energy_momentum_perfect} to
\begin{align}
T_{\alpha\beta} := \Density u_\alpha u_\beta + (p+\Pi)\Proj_{\alpha\beta},
\label{E:Energy_momentum}
\end{align}
where $\Pi$ is the bulk viscosity, which encodes viscous contribution to the pressure.
We remark that here we restrict ourselves to discuss the generalization of \eqref{E:Energy_momentum_perfect}
to viscous theories of bulk viscosity. Other dissipative effects, such as shear viscosity or heat conduction,
can also be considered; see the above references for a discussion of this more general situation.
In the theories of Eckart and Landau, the bulk viscosity is defined to be  
\begin{align}
	\label{E:Eckart}
    \Pi := - \zeta \nabla_\alpha u^\alpha,
\end{align}
where $\zeta$ is a known function of $\Density$ and $n$. $\zeta$ is called the bulk
viscosity coefficient.
The equations of motion are still given by \eqref{E:Div_T}, with $u$ satisfying \eqref{E:Normalization_u}, and 
possibly the conservation of baryon number \eqref{E:Conservation_n_initial}
if $p=p(\Density,n)$.
More precisely, equation \eqref{E:Conservation_n_initial} still holds if $p=p(\Density)$ only, but
in this case the conservation of baryon number decouples from the rest of the system
and can be integrated separately.
 Observe that, as in the case of the perfect
fluid, we can project the divergence of the energy-momentum tensor in the directions parallel and perpendicular
to $u$, leading to a closed system of equations.

The choice \eqref{E:Eckart} is motivated by thermodynamic considerations and
leads to a covariant version of the Navier-Stokes equations. Formally, the non-relativistic limit
of the corresponding evolution reduces to the (compressible) Navier-Stokes equations
\cite{Marcelo2017}. However, the equations of motion have a \emph{parabolic character} \cite{PichonViscous} and are thus incompatible with the most basic requirement of relativity, namely, causality,
which requires finite propagation speed \cite{Hiscock_Lindblom_acausality_1987}. Moreover, solutions of the Eckart model can exhibit catastrophic instabilities, see e.g. \cite{Hiscock_Lindblom_stability_1983, Hiscock_Lindblom_instability_1985, Hiscock_Lindblom_pathologies_1988}. 
The attempt to find \emph{hyperbolic} equations of motion that lead to causal and stable theories
has guided much of the research on relativistic viscous fluids. We refer the reader to the previous
references for a discussion on the history of relativistic viscous fluids and the several attempts
to construct causal and stable theories. We will next discuss the 
M\"uller-Israel-Stewart theory that is the focus of this work.

\subsection{The M\"uller-Israel-Stewart theory}\label{S:MIS}
In the M\"uller-Israel-Stewart theory, the viscous contributions are \emph{not} given as a function of the hydrodynamic variables $u$, $\Density$, and $n$ and its derivatives, as it was the case, for example, in Eckart's theory where $\Pi$ is defined by \eqref{E:Eckart}.
Rather, the viscous contributions are taken to be new dynamic variables on their own right.
In the case of a theory with only bulk viscosity, one again starts with 
the energy momentum tensor \eqref{E:Energy_momentum}, 
which is conserved in the sense of \eqref{E:Div_T}, and with $u$ normalized according to \eqref{E:Normalization_u}. But now $\Pi$ is a new dynamic variable that satisfies the equation
\begin{align}\label{E:Eq_Pi_truncated}
\tau_0 u^\al \nabla_\al \Pi + \Pi + \zeta \nabla_\al u^\al = 0,
\end{align}
where $\tau_0$ is the fluid's relaxation time coefficient
or, more precisely, the relaxation time associated with bulk viscosity ($\zeta$ is 
as above, i.e., the fluid's bulk viscosity coefficient). Since we are not
considering shear viscosity or heat flow (for which there would be further relaxation times
associated), we refer to $\tau_0$ simply as the relaxation time coefficient of the fluid.
Both $\zeta$ and $\tau_0$ are known functions of $\Density$ and $n$. More generally,
one could consider them to be functions of $\Pi$ as well, but we will not assume so in this work, as it would make the analysis more cumbersome.

In the original works of M\"uller, Israel, and Stewart \cite{MIS-1, MIS-2, MIS-3, MIS-4, MIS-5}, 
an equation similar to \eqref{E:Eq_Pi_truncated} was adopted, motivated
by both thermodynamic considerations and the desire to construct a causal theory.
Regarding the former, M\"uller, Israel, and Stewart's choices were
used to show that
a suitably defined notion of non-equilibrium entropy is non-decreasing along the flow, i.e., the 
second law of thermodynamics is satisfied.
Regarding the latter, the idea is that causality requires a relaxation mechanism that 
allows the system to relax to equilibrium. Relaxation-type dynamics of the type \eqref{E:Eq_Pi_truncated} has a long tradition in physics, starting from the seminal work of Cattaneo \cite{Cattaneo1958}. In modern approaches to relativistic viscous fluids, equation
\eqref{E:Eq_Pi_truncated} is derived form kinetic theory \cite{Denicol:2012cn} 
(see also \cite{degroot}) or from effective theory arguments \cite{Baier:2007ix}.

Observe that upon setting $\tau_0 = 0$, one recovers Eckart's choice \eqref{E:Eckart}. 
Thus, it is \emph{precisely} the term $\tau_0 u^\al \nabla_\al \Pi$ which makes the equations hyperbolic and ensures finite speed of propagation. We stress, however, that while obtaining a stable
and causal theory was one of the main motivations for the construction of the 
M\"uller-Israel-Stewart theory, it was not until very recently, with the works
\cite{DisconziBemficaHoangNoronhaRadoszNonlinearConstraints, BemficaDisconziNoronha_IS_bulk},
that it was in fact established that the M\"uller-Israel-Stewart theory leads to causal
equations of motion (stability was proved in \cite{Hiscock_Lindblom_stability_1983,Olson:1989ey}).
In fact, a common misconception is that the M\"uller-Israel-Stewart theory was proved to be causal
a long time ago in the works \cite{Hiscock_Lindblom_stability_1983,Olson:1989ey}. These works
only show the causality of the equations linearized around thermodynamic constant equilibrium states.
I.e., they consider the linearization around constant $u$, $\Density$, $n$, and $\Pi=0$
and proceed to show that the resulting \emph{linear equations} are causal. More precisely,
these works consider the equations also with shear viscosity and heat conduction and linearize
the equations around states where not only $\Pi$, but the other dissipative contributions also vanish, and the remaining variables are constant.


In this work, instead of \eqref{E:Eq_Pi_truncated}, we will take a slightly more general  
type of relaxation law for $\Pi$, namely
\begin{align}
\tau_0 u^\al \nabla_\al \Pi + \Pi +\la \Pi^2 +\zeta \nabla_\al u^\al = 0,
\label{E:Eq_Bulk_lambda}
\end{align}
where $\lambda$ is a transport coefficient associated with the nonlinear behavior of $\Pi$,
which is a known function of $\Density$ and $n$. Equation \eqref{E:Eq_Bulk_lambda}
obviously reduces to \eqref{E:Eq_Pi_truncated} when $\lambda=0$ and our result
also covers this case.
The inclusion of $\lambda$
is motivated by kinetic theory, because derivations from kinetic theory lead
to a form of the equations more general than originally proposed by M\"uller, Israel, and Stewart, including
the presence of the term $\la \Pi^2$, see \cite{Denicol:2012cn}. 
More generally, such arguments
from kinetic theory suggest also the inclusion of the term $\zeta \Pi \nabla_\alpha u^\alpha$ 
on the LHS of \eqref{E:Eq_Bulk_lambda}. We do not consider such a term here for simplicity,
as the proof for \eqref{E:Eq_Pi_truncated} and \eqref{E:Eq_Bulk_lambda} is essentially the same,
whereas including $\zeta \Pi \nabla_\alpha u^\alpha$ would require further conditions and analysis.

\subsection{The equations of motion}
We are now ready to summarize the equations of motion to be studied. Once again, we consider
the projection of the energy-momentum tensor \eqref{E:Energy_momentum} onto the directions parallel and perpendicular to $u$. We have:
\begin{align}
    &u^\al \nabla_\al \Density + (\Density+p+\Pi) \nabla_\al u^\al = 0, \label{E:Conservation_energy}\\
    &(\Density + p+ \Pi) u^\beta \nabla_\beta u_\al + \Proj_\al^\beta \nabla_\beta(p+\Pi) = 0,
    \label{E:Conservation_momentum}\\
    &u^\al \nabla_\al n + n \nabla_\al u^\al = 0,\label{E:Conservation_n}\\
    &\tau_0 u^\al \nabla_\al \Pi + \Pi +\lambda \Pi^2 +\zeta \nabla_\al u^\al = 0. \label{E:Bulk_equation}
\end{align}
For the reader's convenience, we recall here the character of the several quantities 
appearing in \eqref{E:Conservation_energy}-\eqref{E:Bulk_equation}. 
The quantities $\Density$, $n$,  and $\Pi$ are real-valued functions defined 
in $\mathbb{R}^{1+3}$ or a subset of it (e.g., $[0,T_0)\times \mathbb{R}^3$ for some $T_0>0$).
The (four-)velocity $u$ is a vector field on $\mathbb{R}^{1+3}$ (or a subset of it) that
is timelike, future directed, and normalized by \eqref{E:Normalization_u}.
The pressure $p$, the bulk viscosity coefficient$\zeta$, the relaxation time $\tau_0$, and the transport
coefficient $\lambda$ are known functions of $\Density$ and $n$. In particular, 
the relation $p = p(\Density,n)$ is called an equation of state. 

\begin{definition}[The M\"uller-Israel-Stewart equations]
\label{D:MIS_equations}
The M\"uller-Israel-Stewart equations 
with bulk viscosity and no shear viscosity nor heat flow, hereafter referred to simply as the
M\"uller-Israel-Stewart equations, are given by 
\eqref{E:Conservation_energy} -- \eqref{E:Bulk_equation}. 
\end{definition}

Note also that the M\"uller-Israel-Stewart equations in the form \eqref{E:Conservation_energy}
-- \eqref{E:Bulk_equation} are \emph{not} a system of conservation laws.

We have not listed \eqref{E:Normalization_u} among the 
equations of motion because such normalization is better understood 
as a constraint that is propagated by the flow, i.e., a condition that 
holds for $t>0$ if it holds initially. This can be seen
by contracting \eqref{E:Conservation_momentum} with $u^\alpha$. We remark that we will
only consider data for which the normalization \eqref{E:Normalization_u} holds, 
see Definition \ref{D:Initial_data}.

The compressible Navier-Stokes equations are recovered as a suitable (formal) limit of the 
M\"uller-Israel-Stewart equations, as follows. First, one considers the limit of small gradients and small deviations from equilibrium, wherein
$\nabla_\alpha \Pi \sim 0$ and $\Pi^2 \sim 0$. In this situation, we drop these
quantities from \eqref{E:Bulk_equation}. Then $\Pi$ is given by \eqref{E:Eckart}, i.e., one recovers
the Eckart theory. The non-relativistic limit then produces the compressible Navier-Stokes equations,
as already mentioned.

From the point of view of the Cauchy problem, we are given the values of $\Density$, $n$,
$u$, and $\Pi$ on a Cauchy surface of $\mathbb{R}^{1+3}$, which for simplicity we take
to be $\Sigma_0 := \{ t = 0\}$. In view \eqref{E:Normalization_u}, is suffices to prescribe
as data the components of $u$ tangent to $\Sigma_0$. This leads to the following
definition.

\begin{definition}[Initial-data sets for the  M\"uller-Israel-Stewart equations]
\label{D:Initial_data}
An initial-data set
for the M\"uller-Israel-Stewart equations consists of real-valued functions
$\mathring{\Density}, \mathring{n}, \mathring{\Pi} \colon \Sigma_0 \rightarrow \mathbb{R}$
and a vectorfield $\mathring{\mathbf{u}} \colon \Sigma_0 \rightarrow \mathbb{R}^3$.
We denote by $\mathring{u} \colon \Sigma_0 \rightarrow \mathbb{R}^4$ the
initial four-velocity determined from $\mathring{\mathbf{u}}$ with the help of \eqref{E:Normalization_u}.
\end{definition}

Given an initial-data set, one then seeks to determine functions $\Density$, $n$, $\Pi$, and a vector field $u$ that satisfy the
M\"uller-Israel-Stewart equations in a neighborhood of $\Sigma_0$ and take the given
initial data on $\Sigma_0$. To be more precise, for $u$ we require that its projection
onto the tangent space of $\Sigma_0$ agrees with $\mathring{\mathbf{u}}$.
Furthermore, in view of the preceding discussion, we are interested in finding solutions
that are causal (see Definition \ref{D:Causality} for the precise definition of causality).

Under suitable hypotheses on the data, \emph{satisfied by the initial data in 
Theorem \ref{T:Main_theorem_blowup_simpler},} the local existence,
uniqueness, and causality of solutions to the Cauchy problem
for the M\"uller-Israel-Stewart equations has been established in \cite{BemficaDisconziNoronha_IS_bulk}. 
 
\subsection{Causality\label{S:Causality}} 
While the finite speed of propagation property is usually automatic for hyperbolic equations of motion, the requirement of causality places additional constraints on the fluid description, as one
wants not only that information propagates at finite speed but also that the speeds of propagation
are at most the speed of light. To be more precise, causality means the following.

\begin{definition}[Causality]
\label{D:Causality}
Consider in $\mathbb{R}^{1+3}$ a system of partial differential equations for an unknown $\psi$, which
we write as $\mathfrak{P} \psi = 0$, where $\mathfrak{P}$ is a differential operator which is allowed to depend
on $\psi$. Let $\Sigma$ be a Cauchy surface, $\varphi_0$ Cauchy data for $\mathfrak{P} \psi = 0$ given
along $\Sigma$, and $\varphi$ a solution to the corresponding Cauchy problem defined in a neighborhood
$\mathcal{O}$ of $\Sigma$. We say that the system is causal if for any $P \in \mathcal{O}$ in the future
of $\Sigma$, $\varphi(P)$ depends only on the Cauchy data on $J^-(P) \cap \Sigma$, where
$J^-(P)$ is the past-directed light cone with apex $P$ (see Figure \ref{fig1}).
\end{definition}

We have given the definition of causality that suffices to our purposes, i.e., for equations in Minkowski space,
but this can be generalized to arbitrary globally hyperbolic space-times, see, e.g., \cite{HawkingEllisBook}.
We refer the reader to \cite{DisconziBemficaNoronhaConformal} for further discussion on causality
of relativistic viscous fluid theories. For the physical importance of causality and its relation to global
hyperbolicity, see \cite{Witten:2019qhl}. 

\begin{figure}[h]
\begin{center}
\includegraphics[width=0.6\textwidth]{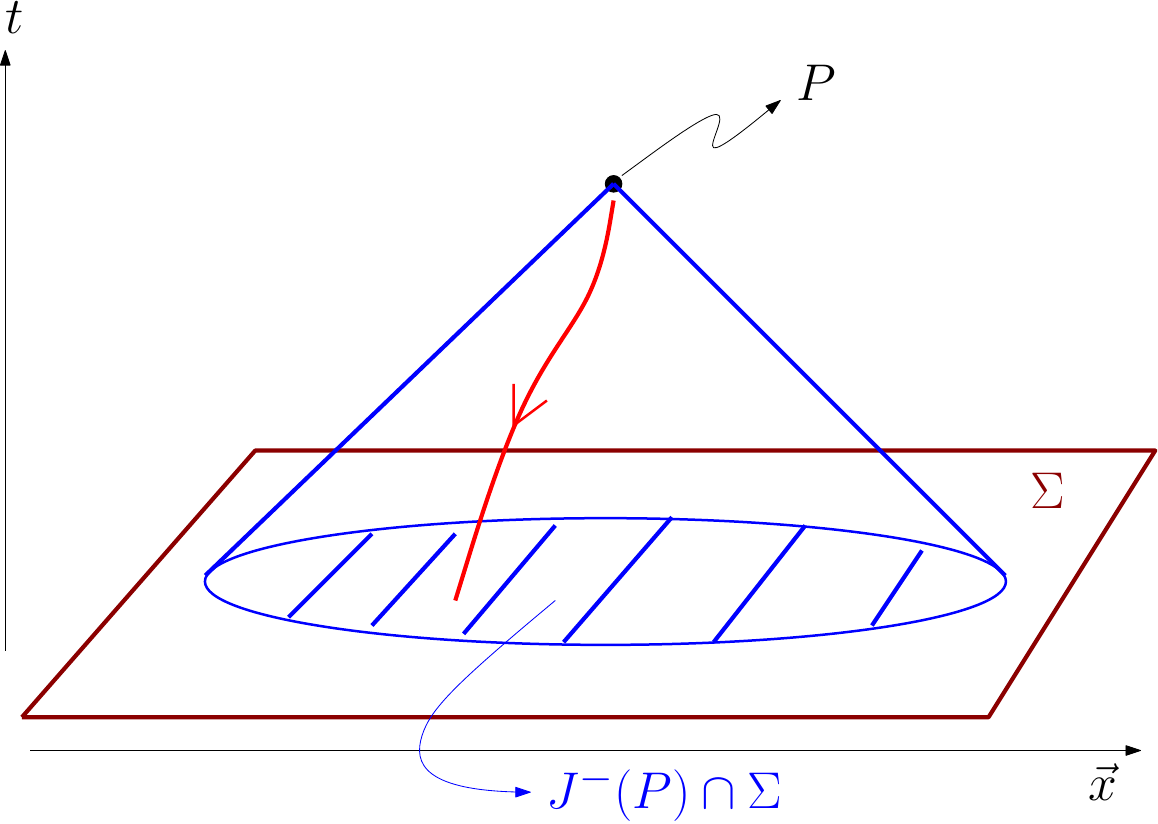}
\caption{Illustration of causality.  $J^-(P)$ 
is the past light-cone with vertex at $P$.
Points inside $J^-(P)$ can be joined to a point $P$ in space-time by a causal past directed curve (e.g. the red line). 
The value of solutions at $P$ depends only on the data on $J^-(P) \cap \Sigma$.
}
\label{fig1}
\end{center}
\end{figure}

Causality is intrinsically related to the characteristics of the equations of motion. The characteristics
of the M\"uller-Israel-Stewart equations were computed in \cite{BemficaDisconziNoronha_IS_bulk}. Aside
from  the flow lines of $u$, the characteristics of the M\"uller-Israel-Stewart equations
correspond to null-hypersurfaces of an acoustical metric \cite{DisconziSpeckRelEulerNull}
with speed
\begin{align}
\label{eq_sound_speed}
c_s^2 = c_s^2(\Density,n,\Pi) := \frac{\zeta}{\tau_0 (\Density+p+\Pi)} + \del_{\Density} p + \frac{n \del_n p}{\Density+p+\Pi},
\end{align} 
whenever the right-hand side is non-negative. 
A necessary and sufficient condition for causality is simply $c_s^2 \leq 1$ (see \cite{BemficaDisconziNoronha_IS_bulk}).

Observe that $c_s^2$ reduces to the sound speed of a perfect fluid when $\zeta=0=\Pi$. Thus, 
\eqref{eq_sound_speed} should be viewed as a generalization of the sound speed in the
presence of viscous effects. On the other hand, equation \eqref{eq_sound_speed} also shows that viscous
effects directly contribute to the system's characteristics and, therefore, viscous effects cannot
be viewed as a perturbation of the perfect fluid case, even when $\zeta$ and $\Pi$ are small.
This is already clear from equations \eqref{E:Conservation_energy}-\eqref{E:Bulk_equation} in that
$\Pi$ contributes to the principal part of the system. 

 {\subsection{Energy condition\label{S:Dominant}}
The weak energy condition plays a role in our work, thus we recall its definition here.
\begin{definition}[Weak energy condition]
\label{D:WeakEnergy}
The energy-momentum tensor \eqref{E:Energy_momentum} is said to satisfy the
weak energy condition if $T^{\alpha\beta} v_\alpha v_\beta\geq 0$ for any timelike co-vector $(v_\al)$. We say that a solution $(\Density,n,\Pi,u)$ to the M\"uller-Israel-Stewart equations satisfies the weak energy condition if the corresponding 
energy-momentum tensor \eqref{E:Energy_momentum} satisfies the weak energy condition.
\end{definition}
For the energy-momentum tensor \eqref{E:Energy_momentum}, the weak energy condition is equivalent to the inequality
\begin{align}
    \Density+p+\Pi \geq 0;
\end{align}
a very sensible requirement in view of the momentum equation \eqref{E:Conservation_momentum}. Since $\rho+p+\Pi$ determines the inertia of individual fluid elements, $\rho+p+\Pi < 0$ would correspond to a negative inertia. We will show below that the weak energy condition is propagated by our viscous fluid equations, provided it is satisfied at $t=0$.}

\subsection{Admissible solutions}\label{S:Admissible} In this Section we introduce a class of solutions that satisfy some basic physical and mathematical requirements forming what we will take as \emph{admissible} solutions. 

Our admissible solutions are essentially those that have
enough regularity, are causal  {and for which the pressure is bounded from below by a negative constant}.  It will be implicit that constitutive relations $p =p(\Density,n)$,
$\zeta = \zeta(\Density,n)$, $\tau_0 = \tau_0(\Density,n)$, and $\lambda = \lambda(\Density,n)$
are given functions whenever definitions involve these quantities.

 {\begin{definition}[Physical states]
\label{D:Physical_states}
The set of physical states $\mathcal{P}$ for the M\"uller-Israel-Stewart theory 
is the set of $(\Density, n, \Pi)\in \mathbb{R}^3$ satisfying the conditions
\begin{enumerate}
    \item $\Density > 0, n > 0$ (positive energy density and number density)
    \item  $0 < c_s^2(\Density, n, \Pi) < 1$ (strict causality)
\end{enumerate}
\end{definition}}

\begin{remark}
Above, and in much of what follows, we will consider defining properties where strict inequality holds.
We discuss this choice in Section \ref{S:Significance}.
\end{remark}

Next, we define an admissible solution as follows.

\begin{definition}\label{D:physical_solution_of_MIS}
A solution $(\Density, n, \Pi, u) \in (C^1([0, T_0)\times \R^3))^{7}$ 
to the M\"uller-Israel-Stewart equations is called \emph{admissible} if states are physical in the sense of definition \ref{D:Physical_states}, i.e., if
\begin{align}
    (\Density(t, x), n(t, x), \Pi(t, x))\in \mathcal{P}
\end{align}
holds for all $(t, x)\in [0, T_0)\times \R^3$. An $C^1$-initial data set $(\mathring{\Density}, \mathring{n}, \mathring{\Pi}, \mathring{\mathbf{u}})$ is called admissible if $(\mathring{\Density}(x), \mathring{n}( x), \mathring{\Pi}(x))\in \mathcal{P}$ for all $x\in \mathbb{R}^3$.
\end{definition}

Next, we state our  assumptions on the constitutive functions $p, \zeta, \tau_0$, and $\lambda$.

\begin{assumption}[Assumptions on $p$, $\zeta$, $\tau_0$, and $\lambda$]
\label{A:Assumptions_on_pressure_and_viscosity} 
We consider constitutive relations for the pressure $p(\Density, n)$, bulk viscosity coefficient 
$\zeta(\Density, n)$, relaxation time coefficient 
$\tau_0(\Density, n)$, and the transport coefficient $\lambda(\Density,n)$ that satisfy:
 {
\begin{enumerate}
    \item[(A1)] The functions $p, \zeta, \tau_0$ are defined on $\R^+\times \R^+$, are
    smooth, and admit smooth extensions to $(\Density, n)\in \R\times \R^+$. Moreover, there exist constants $p_0, p_1\geq 0$ such that
    \begin{align}
        &-\Density \leq p(\Density, n) \leq \Density+p_1\\
        &-p_0 < p(\Density, n)
    \end{align}
    for all $(\Density, n)\in \mathcal{P}$. 
    \item[(A2)] $p(\Density, n)$ is globally Lipschitz on $\R\times\R^+$ and
    \begin{align}
        \del_\Density p(\Density, n) \neq 0, ~~\del_n p(\Density, n) \neq 0
    \end{align}
    for $(\Density, n)\in \mathcal{P}$.
    \item[(A3)] The functions $\zeta, \tau_0 : \R\times\R^+ \to \R$ are smooth, 
	$\zeta \geq 0$, and $\tau_0>c > 0$ for some $c>0$. 
     In addition, 
\begin{align}
        \sup_{(\Density, n)\in \R\times\R^+}\left[\left\vert\del_n \left(\frac{\zeta(\Density, n)}{\tau_0(\Density, n)}\right)\right\vert+\left\vert\del_\Density\left( \frac{\zeta(\Density, n)}{\tau_0(\Density, n)}\right)\right\vert\right] < \infty
\end{align}  
    holds. Moreover, 
    \begin{align}
        \del_\Density \left(\frac{\zeta}{\tau_0}\right)  \geq 0 
    \end{align}
    for all $(\Density, n)\in \R\times \R^+$.
    \item[(A4)] We have
    \begin{align}\label{eq_assumption_for_zeta_both_rho_and_n}
        &\int_{0}^{\infty} \frac{1}{n}\sup_{\Density \in\R}\left\vert\frac{\zeta(\Density, n)}{\tau_0(\Density, n)}\right\vert~dn < \infty. 
     \end{align}   
    \item[(A5)] $\lambda: \R\times \R^+\to \R$ is either a smooth positive function ($\lambda(\Density, n) > 0$) with 
    \begin{align}
        p(\Density,n)+\Density < \frac{1}{\lambda(\Density, n)}.
    \end{align}
    for all $(\Density, n)\in \mathcal{P}$, or $\lambda(\Density, n) = 0$ for all $(\Density, n)$.
\end{enumerate}}
\end{assumption}

\begin{remark}[Assumption \ref{A:Assumptions_on_pressure_and_viscosity} is not empty]\label{remark2}
We remark that there exist functions $p$, $\zeta$, $\tau_0$, and $\lambda$ satisfying
Assumption \ref{A:Assumptions_on_pressure_and_viscosity}. As a simple example, take $\tau_0 > 0$ to be a constant and $\zeta = \zeta(n)$ such that
\begin{align*}
    0 < \zeta(n) \leq C n^{\delta_1} \text{~~~for small $n$}\\
    0 < \zeta(n) \leq C n^{-\delta_2} \text{~~~for large $n$} 
\end{align*}
for some $\delta_1, \delta_2 > 0$.
 It might seem unusual to require that the pressure can be extended to a function defined for all values $(\Density, n)\in \R\times \R^+$, including negative densities. This poses no real problems and is done for a clearer presentation. A crucial ingredient in our proof of breakdown of smooth solutions is the construction of an a-priori bound for the bulk viscous pressure $\Pi$. To keep this construction transparent, we assume that $p$ and the constitutive functions $\zeta, \tau_0$ are extendable to the domain $(\Density,n)\in \R\times\R^+$. To show how this extension works in practice, we consider the ideal gas equation of state (see \cite{RezzollaZanottiBookRelHydro})
    \begin{align}\label{E:ideal_gas_1}
        p(\epsilon, n) = m n \epsilon (\gamma_{\text{ad}} - 1) 
    \end{align}
    where $m$ is the mass per particle and $\epsilon$ is the \emph{specific internal energy} satisfying
    \begin{align}\label{E:specific_internal_energy}
        \Density = m n (1 + \epsilon).
    \end{align}
$\gamma_{\text{ad}} > 1$ is the adiabatic index of the fluid. Combining \eqref{E:ideal_gas_1} and \eqref{E:specific_internal_energy}, we arrive at the pressure 
    \begin{align}
        p(\Density, n) = (\gamma_{\text{ad}} - 1) (\Density - m n),
    \end{align}
    which can clearly be regarded as a function on $(\Density, n)\in \R\times\R^+$. Note that $(\Density, n)\in \mathcal{P}$ requires $\Density > m n$. 

\end{remark}

\subsection{Statement of the results} 
\label{S:Results}
We are now ready to state our results. Our main Theorem states that there exists initial data for which the corresponding solutions break down in finite time.
\begin{theorem}[Breakdown of admissible solutions to the M\"uller-Israel-Stewart equations]
\label{T:Main_theorem_blowup_simpler}
Consider the M\"uller-Israel-Stewart equations and suppose that Assumption
\ref{A:Assumptions_on_pressure_and_viscosity} holds.  
Let $R_0>0$, $\bar{\Density}>0$, and $\bar{n} >0$ be constants and let 
$c := c_s(\bar \Density, \bar n, 0) $. Assume moreover $0<c < 1$. 
Then, there exists smooth admissible initial data
$(\mathring{\Density},\mathring{n},\mathring{\Pi},\mathring{\mathbf{u}})$
for the M\"uller-Israel-Stewart equations with the following properties.

1)  $(\mathring{\Density},\mathring{n},\mathring{\Pi},\mathring{\mathbf{u}}) = 
(\bar{\Density},\bar{n},0,0)$ outside $B_{R_0}$.

2) There exists a $T_0>0$ and a unique smooth admissible solution
$(\Density,n,\Pi,u)$ to the M\"uller-Israel-Stewart equations defined for $t \in [0,T_0)$ and
taking the data $(\mathring{\Density},\mathring{n},\mathring{\Pi},\mathring{\mathbf{u}})$.
For each $0\leq t < T_0$, this solution satisfies 
\begin{align}
    \Density(t, x) = \bar \Density, ~n(t, x) = \bar n, ~\Pi(t, x) = 0, ~\mathbf{u}(t, x) = 0.  
      \nonumber
\end{align}
outside a ball of radius $R(t) = R_0 + c t$  {and the \emph{weak energy condition} holds for all $(t, x)\in [0, T_0)\times \R^3$, i.e.
\begin{align}
    \Density+p+\Pi \geq 0. 
\end{align}}

3) The solution $(\Density,n,\Pi,u)$  cannot be continued as a $C^1$ admissible solution
past $t=T_0$.  {More precisely, we have 
\begin{align}\label{E:C1}
\lim_{t\to T^-_0} \left(\vert (n(t,\cdot), \Density(t,\cdot), \Pi(t,\cdot))\vert _{C^1(\R^3)} +  \vert \nabla\mathbf{u}(t,\cdot)\vert _\infty\right)= \infty
\end{align} 
or the values $(\Density, n, \Pi)$ leave any compact set $\Omega$ contained in $\mathcal{P}$.}
\end{theorem}

Theorem \ref{T:Main_theorem_blowup_simpler} thus states that $(\Density,n,\Pi,u)$
becomes singular in the sense that it 
cannot be continued as a classical (i.e., $C^1$) solution all the way up to $t=T_0$, or 
$(\Density,n,\Pi,u)$ is well-defined at $t=T_0$ but no longer satisfies the physical conditions
that define an admissible solution. This happens inside the ball of radius $R_0 + c T_0$, whereas
outside this ball the fluid is in equilibrium, see Figure \ref{F:Equilibrium_outside}.

\begin{figure}[h]
\begin{center}
\includegraphics[width=0.6\textwidth]{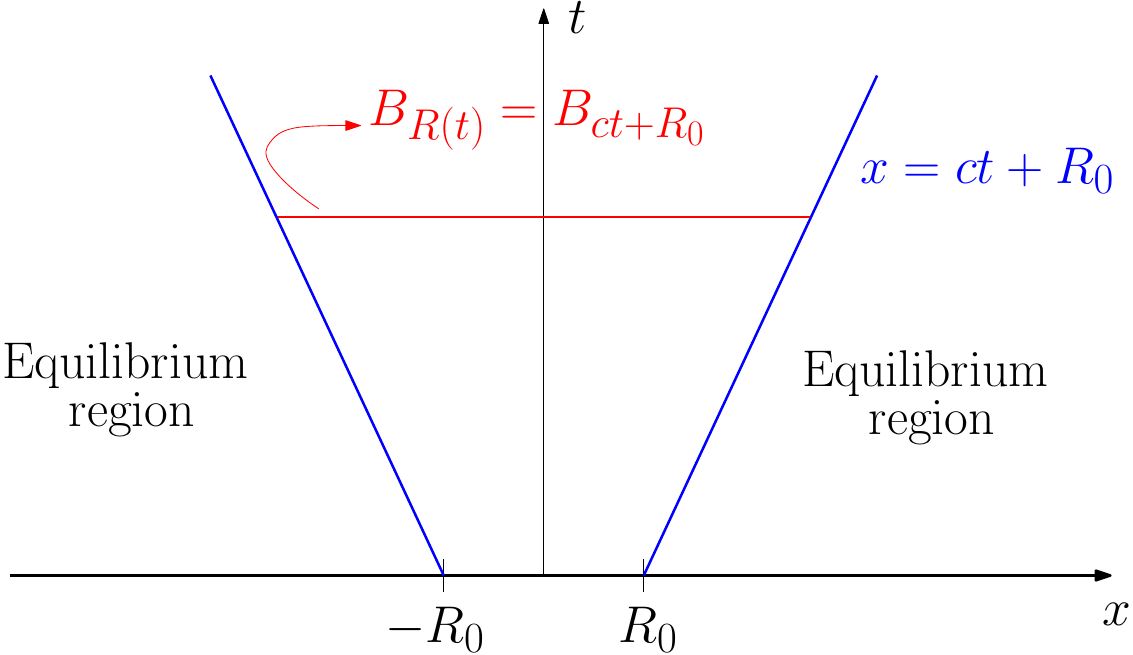}
\caption{Illustration of the space-time regions of Theorem \ref{T:Main_theorem_blowup_simpler}.
For each time $t$, the solution remains in equilibrium outside the ball of radius $R(t) = ct+R_0$.
Our illustration is in $1+1$ dimensions for simplicity, but the result holds in $3+1$ dimensions.
}
\label{F:Equilibrium_outside}
\end{center}
\end{figure}

The only constraint on the constants $R_0$, $\bar{\Density}$, and $\bar{n}$ 
is that they are positive and $c = c_s(\bar \Density, \bar n, 0) < 1$. 
Thus, we obtain a family
of initial data parametrized by $R_0,\bar{\Density},\bar{n}$. 
In particular, $R_0>0$ can be taken very small, so our data is a localized modification of a constant state.
This is the sense in which we consider $(\mathring{\Density},\mathring{n},\mathring{\Pi},\mathring{\mathbf{u}})$
a perturbation of a constant state. We are not, however, claiming that 
$(\mathring{\Density},\mathring{n},\mathring{\Pi},\mathring{\mathbf{u}})$ is close to 
a constant state in the $C^0$ topology. In fact, the initial data we will
construct will have large velocities $\mathring{\mathbf{u}}$.
We illustrate the behavior of our initial data in Figure \ref{F:Initial_data}.

\begin{figure}[h]
\begin{center}
\includegraphics[width=0.6\textwidth]{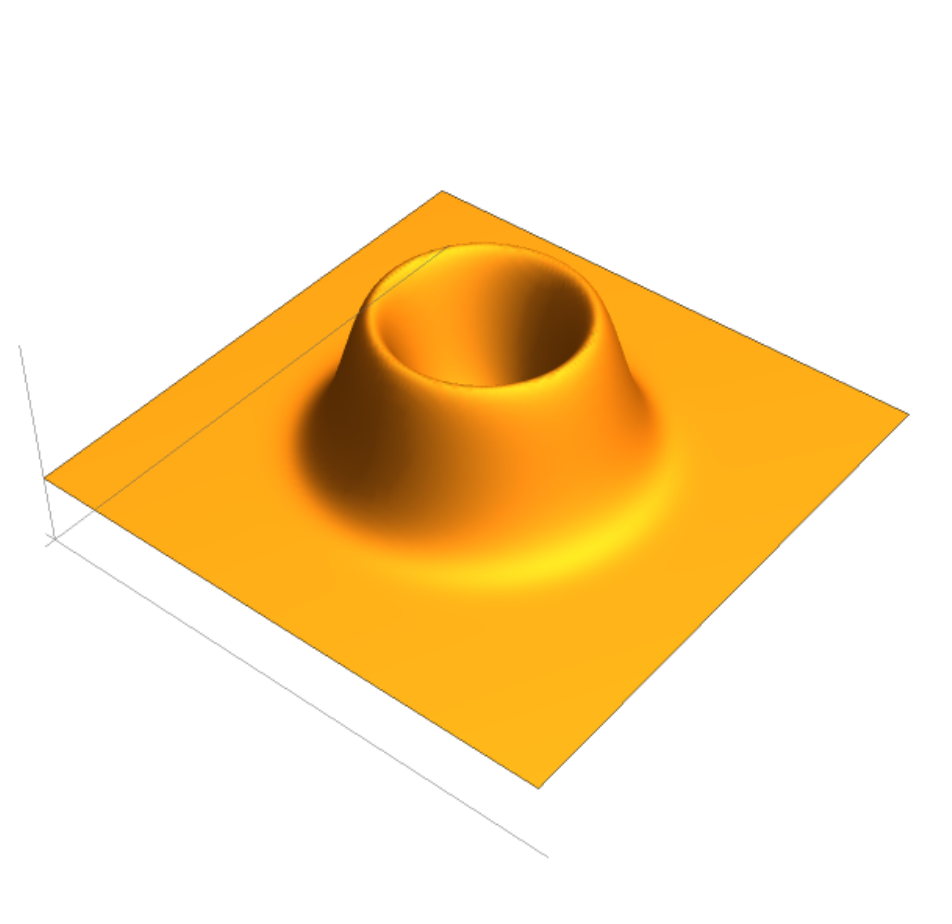}
\caption{Illustration of the initial data in Theorem \ref{T:Main_theorem_blowup_simpler}. The figure illustrates $\vert\mathbf{u}\vert$. The velocity is zero, corresponding to equilibrium, outside the ball of radius $R_0$. Note that the figure shows a radially symmetric configuration for simplicity. The class of initial profiles for $\mathbf{u}$ leading to breakdown is characterized by an integral inequality and does not require any symmetry.
}
\label{F:Initial_data}
\end{center}
\end{figure}

The next Theorem states
that this family of data is stable under perturbations.

\begin{theorem}[Stability of the breakdown of admissible solutions]
\label{T:Theorem_stability}
Suppose that Assumption \ref{A:Assumptions_on_pressure_and_viscosity} holds and 
let $(\mathring{\Density},\mathring{n},\mathring{\Pi},\mathring{\mathbf{u}})$
be initial data for the M\"uller-Israel-Stewart equations given by Theorem \ref{T:Main_theorem_blowup_simpler}.
There exists an $\varepsilon>0$ with the following properties. If 
$(\tilde{\Density},\tilde{n},\tilde{\Pi},\tilde{\mathbf{u}})$ is initial data for the
M\"uller-Israel-Stewart equations such that $\tilde{\Density} = \bar{\Density}$,
$\tilde{n} = \bar{n}$ outside $B_{R_0}$, where $R_0$, $\bar{\Density}$, and $\bar{n}$ are as in
Theorem \ref{T:Main_theorem_blowup_simpler}, and
\begin{align}
\vert(\mathring{\Density},\mathring{n},\mathring{\Pi},\mathring{\mathbf{u}})(x) - 
(\tilde{\Density},\tilde{n},\tilde{\Pi},\tilde{\mathbf{u}})(x) \vert < \varepsilon
\nonumber
\end{align}
for all $x\in \mathbb{R}^3$, then the conclusions of Theorem \ref{T:Main_theorem_blowup_simpler}
hold for $(\tilde{\Density},\tilde{n},\tilde{\Pi},\tilde{\mathbf{u}})$.
\end{theorem}

The main idea behind Theorem \ref{T:Main_theorem_blowup_simpler} is the following. The breakdown is driven by the ideal component and the main task then is to control the evolution of $\Pi$. The breakdown of the ideal part, in turn, follows the basic philosophy of the previous work \cite{Shadi_YanGuo1999}, although our proof is more refined, allowing us to treat all values of the sound speed. Hence our techniques can also be employed to obtain a new breakdown result for the relativistic Euler equations, which we recall are
given by
\begin{align}
    &u^\al \nabla_\al \Density + (\Density+p) \nabla_\al u^\al = 0, \label{Ep:Conservation_energy}\\
    &(\Density + p) u^\beta \nabla_\beta u_\al + \Proj_\al^\beta \nabla_\beta p = 0,\label{Ep:Conservation_momentum}\\
    &u^\al \nabla_\al n + n \nabla_\al u^\al = 0.\label{Ep:Conservation_n}
\end{align}
For the relativistic Euler equations, the sound speed is given by
\begin{align}
\label{eq_sound_speed_perfect}
c_{\text{Euler}}^2 = c_{\text{Euler}}^2(\Density,n) := \del_{\Density} p + \frac{n \del_n p}{\Density+p}.
\end{align} 
With \eqref{eq_sound_speed_perfect} we can define admissible solutions for the relativistic Euler equations in the same way as in Definition \ref{D:physical_solution_of_MIS}.

\begin{theorem}[Breakdown of admissible solutions for the relativistic Euler equations]
\label{T:Main_theorem_Euler}
Consider the relativistic Euler equations and assume that the equation of state $p=p(\Density,n)$
satisfies the following:
$p$ is a smooth function defined on an open set $\mathcal{P}\subset (\Density, n)\in \R^+\times \R^+$,
$p(\Density, n)$ is globally Lipschitz on $\mathcal{P}$ and satisfies
\begin{align}\label{pressurecondition}
    -\Density \leq p(\Density, n)\leq \Density+p_1
\end{align}
as well as $p(\Density, n)\geq -p_0$ on $\mathcal{P}$ for positive constants $p_0, p_1$. Assume that the perfect sound speed
$c_{s, \text{Euler}}$ satisfies $0 < c_{s, \text{Euler}}^2 < 1$ as well as $\del_\Density p \neq 0, \del_n p\neq 0$ on $\mathcal{P}$.
Let $R_0>0$, $\bar{\Density}>0$, and $\bar{n}>0$ be constants. 

Then, there exists smooth admissible initial data 
$(\mathring{\Density},\mathring{n},\mathring{\mathbf{u}})$
for the relativistic Euler equations satisfying
$(\mathring{\Density},\mathring{n},\mathring{\mathbf{u}}) = (\bar{\Density},\bar{n},0)$
outside $B_{R_0}$ such that the corresponding local solution to the 
Cauchy problem with data $(\mathring{\Density},\mathring{n},\mathring{\mathbf{u}})$
cannot exist for all time as a $C^1$ solution with values $(\Density, n)\in \mathcal{P}$. Moreover, such initial data
is stable in the sense of Theorem \ref{T:Theorem_stability}.
\end{theorem}

Theorem \ref{T:Main_theorem_Euler} should be compared with 
the well-known result of  Guo and Tahvildar-Zadeh \cite{Shadi_YanGuo1999}
(which, in turn, can be viewed as a generalization to the relativistic setting of Sideris'
well-known result \cite{Sideris1984} for the non-relativistic compressible Euler equations).
\cite{Shadi_YanGuo1999} requires the restriction $c_{\text{Euler}}^2(\mathring{\Density},\mathring{n})<1/3$, whereas 
our result applies to the full range $0<c_{\text{Euler}}^2(\mathring{\Density},\mathring{n})<1$. As an example for a typical set $\mathcal{P}$, we consider a baryonic perfect fluid without chemical and nuclear reactions. In this case, $\mathcal{P}$ is defined by the inequalities $\Density > 0, n > 0, \Density > m n$, where $m$ is the rest mass per particle. For an example of an equation of state $p(\Density, n)$, see the discussion of the ideal gas \eqref{E:ideal_gas_1} above.

Finally we remark that the pressure condition \eqref{pressurecondition} in the case of perfect fluids follows from the causality condition $0<c^2_{\text{Euler}}<1$ and some standard thermodynamic assumptions. More specifically suppose that 
a thermodynamic entropy $S(\Density, n)$ exists for which
\begin{align}
    T dS = d \eps + p d\left(\frac{1}{n}\right) = \frac{1}{n}d\Density - \frac{\Density+p}{n^2} d n
\end{align}
holds for some $T(\Density, n)> 0$. Here, $\eps \geq 0$ is the \emph{specific internal energy density} which is given by $\Density = n m (1 + \eps)$, $m>0$ being the rest-energy per particle.
Moreover, we assume $p$ to be smooth on $\mathcal{P}$ and uniformly Lipschitz. We will make the standard thermodynamic assumption that $\mathcal{P}$ is diffeomorphic to some domain $\widehat{\mathcal{P}}$ in the $(S, \Density)$-plane and that all relevant thermodynamic functions can be expressed in terms of $(S, \Density)$. Then the following holds: 

\begin{proposition}
Assume that $\Density+p > 0$ on $\mathcal{P}$ and that the perfect fluid sound speed 
\begin{align}
    c_{Euler}^2 = \del_\Density p + \frac{n \del_n p}{\Density+p}
\end{align}
satisfies $0 <  c_{Euler}^2 < 1$ on $\mathcal{P}$ and that for each value of $S$, there exists a value of the energy density $\Density = \Density(S)$ such that
\begin{align}
    \sup_{S} \vert p(S, \Density(S)) \vert <\infty.
\end{align}
Then there exist constants $p_0, p_1\geq 0$ such that
\begin{align}
    -p_0 \leq p(\Density, n)\leq \Density + p_1
\end{align}
holds for all $(\Density, n)\in \mathcal{P}$.
\end{proposition}
\begin{proof}
This simply follows from $\frac{\del p(S, \Density)}{\del\Density} = \del_\Density p(\Density,n) + \del_n p(\Density, n) \left(\frac{dn}{d\Density}\right)_{S=\text{const}}= c_{Euler}^2$ by integrating from $\Density(S)$ to any given value $\Density$, using $0 < c_{Euler}^2 < 1$.
\end{proof}

\subsection{Physical and mathematical significance of our results}
\label{S:Significance}

We now we make a few comments about the physical and mathematical aspects
of our results.
The conditions in Assumption \ref{A:Assumptions_on_pressure_and_viscosity} are
natural enough to include what one would expect in many physical systems.  In particular, the (A1) allows for a very general 
pressure function. (A1) holds for all regular matter subject to the \emph{dominant energy condition}. For perfect fluids with positive pressure, it even allows violation of the dominant energy condition. We note that the dominant energy condition for perfect fluids (see \cite[Section 9.2]{WaldBookGR1984}, \cite{HawkingEllisBook}) is thought to hold for all baryonic matter in in the present universe and that this condition takes the form $\vert p\vert \leq \Density$ for a perfect fluid energy-momentum tensor. In addition it allows also for matter with negative pressure, sometimes considered for exotic matter in inflationary cosmology \cite{WeinbergCosmology}. For a discussion of the remaining assumptions, see Remark \ref{remark2}.

The importance of Theorem \ref{T:Main_theorem_blowup_simpler} lies in being \emph{the very first result of its kind} for any viscous relativistic fluid. This, in particular, opens the door for further studies of  breakdown
of solutions for viscous relativistic fluids (we list some natural follow-up questions at the end of this Section).
Along the same lines, our goal is to obtain a first result that motivates further investigations
on the breakdown of solutions of the M\"uller-Israel-Stewart 
equations while avoiding technicalities. Thus, we tried to to provide the simplest proof we could and in the simplest possible setting. Generalizations to more complex situations are important but left for future work.

We also notice that our techniques do not apply to the case when $\mathring{\Density}$ and $\mathring{n}$ vanish. The vanishing of these quantities would correspond to a vacuum interface,
leading to a free-boundary problem, and currently there are no local existence and uniqueness 
results for the M\"uller-Israel-Stewart equations in such a case. In fact, only very recently 
the relativistic free-boundary Euler equations have been shown to be locally
well-posed \cite{OliynykRelativisticLiquidBodies,DisconziIfrimTataru}; see also
\cite{HadzicShkollerSpeck,JangLeFlochMasmoudi,
OliynykRelativisticLiquidBodies,GinsbergRelEulerApriori,DisconziRemarksEinsteinEuler}.
Neither do we address the case when $n$ vanishes identically (see Remark \ref{R:n_zero_case}).

We now make some remarks on the techniques we use. 
Considering the evolution of the second moment of the energy density $I(t)$ and its derivatives 
$\dot{I}(t)$ and $\ddot{I}(t)$ as a means of showing blowup of solutions is a fruitful idea used in a variety of situations. This technique is related to the virial theorem. For non-relativistic perfect fluids, this was done by T. Sideris in \cite{Sideris1984} and for relativistic perfect fluids breakdown was proven by Y. Guo and A.S. Tahvildar-Zadeh in \cite{Shadi_YanGuo1999}. A Ricatti-type differential inequality is derived, showing the blowup. 

For the viscous relativistic fluid equations considered here, the following significant obstacle arises: the bulk viscous pressure $\Pi$ is not given by an equation of state. We have to find a suitable a-priori estimate for $\Pi$, which evolves according to \eqref{E:Bulk_equation}. In the context of the proof of Theorem \ref{T:Main_theorem_blowup_simpler}, however, \emph{control of $\Pi$ in $L^\infty_t$ via a direct application of \eqref{E:Bulk_equation} is not possible}, because not enough information about the behavior of the expansion scalar $\nabla_\al u^\al$ is available. Our approach to overcome this obstacle is to introduce a suitable transport equation in the variables $(\Density, n, \Pi)$ to bound $\Pi$, arriving at a nontrivial a-priori bound for $\Pi$ (see Lemma \ref{L:estimate_for_Pi}). Our results are not tied to special functional forms of $p, \zeta, \tau_0, \lambda$ nor to special symmetries in the solution. 
Another interesting feature of our analysis is that we \emph{do not use} Ricatti-type differential inequalities, allowing for a breakdown criterion with minimal assumptions on the sound speed. 

To finish this discussion, we turn to a few questions that naturally arise from our results. The first,
and perhaps most immediate question, is about the nature of the breakdown of solutions. 
For our class of initial data, the breakdown of solutions can occur in several different
ways, namely, 
\begin{enumerate}[label=(\roman*)]
\item Forming a singularity, i.e., a breakdown of regularity so that solutions
cannot be continued in a $C^1$ fashion;
\item Violating causality. More precisely, according to our definition of admissible solutions,
(ii) is a violation of the strict form of the corresponding properties.
For example, if $c_s=1$, solutions are still causal, but it does not correspond to an admissible
solution in our sense. For simplicity, we will not consider the borderline cases where equality is achieved.
\end{enumerate}
Of these possibilities, (ii) is soft in the sense that solutions might still be well-defined in a mathematical sense, but are not good candidates for physical solutions: a violation of causality is clearly a no-go for (relativistic) physical theories. 

If the breakdown of solutions is caused by (i), then we can further ask about the
nature of the singularity. Understanding the structure of the singularity is important because, when solutions become singular, one would like to continue them past the singularity in a the form of a weak solution. Typically,
this is possible when the singularity is a shock; for more general singularities, there is little
hope of continuing solutions beyond the singularity in a physically meaningful sense. 
For the M\"uller-Israel-Stewart equations, however, the continuation of solutions
in a weak sense is usually not possible \emph{even when the singularity is a shock,}
as it was shown by Olson and Hiscock \cite{HiscockShocksMIS} 
and Geroch and Lindblom \cite{GerochLindblomCausal}.

It is, therefore, important to better understand the nature of the breakdown in Theorem
\ref{T:Main_theorem_blowup_simpler}. If it is (ii), then one has legitimate physical reasons to discard the solutions thereafter. In view of the difficulties of the M\"uller-Israel-Stewart
theory to describe shocks, a breakdown of the form (i) would impose severe limitations in the ability
of the M\"uller-Israel-Stewart theory to describe the long time dynamics of relativistic
viscous fluids. This would be one of the motivations for considering alternative theories
of relativistic fluids with viscosity \cite{TempleViscous}.

One possible way of addressing these questions is to show that the breakdown is in fact of type
(i). The simplest approach would be to construct shock solutions using Riemann invariants
for the system. This, however, \emph{is not possible,}  as we show in Section \ref{S:Riemann}
that Riemann invariants do not exist for the M\"uller-Israel-Stewart equations in $1+1$
dimensions, except for a trivial pressure function.
By itself this does not preclude the construction of solutions with shock singularities,
and powerful techniques for the study of shock formation are available both in $1+1$ dimensions 
(see, e.g., the monographs \cite{BressanBookConservationLaws-2000,SerreConservationLawsBook1,
SerreConservationLawsBook2,DafermosConservationLawsBook,LeFlochShocksBook}) and, more recently,
in higher dimensions (see the review article \cite{SpecketalOverviewShocks} and references therein).
However, is not clear whether such techniques can be applied to the 
M\"uller-Israel-Stewart equations.

Another question of immediate physical interest is to quantify the exact breakdown time $T_1$ given by Theorem \ref{T:Main_theorem_blowup_simpler}. We can give upper bounds on the lifetime of solutions (see Lemma \ref{lemma:blowupQ}). 


\begin{remark}
\label{R:n_zero_case}
One case of great interest in collisions of heavy-ions is when $n=0$. This occurs at high-energies where the number of baryons and anti-baryons is essentially the same. In this situation, we drop equation 
\eqref{E:Conservation_n} and all coefficients and the equation of state become functions of $\rho$ only, i.e., 
$\tau_0 = \tau(\rho), \lambda = \lambda(\rho), \zeta = \zeta(\rho), p = p(\rho)$. This situation falls outside the scope of Theorems \ref{T:Main_theorem_blowup_simpler} and \ref{T:Theorem_stability}. But it is such an important case that one naturally wonders whether our proofs can be modified to treat the $n=0$ case. 
The short answer is that analogues of Theorems \ref{T:Main_theorem_blowup_simpler} and \ref{T:Theorem_stability} for $n=0$ do not seem to be provable by simply following our proofs and ignoring $n$ and equation \eqref{E:Conservation_n} altogether. Some readers might be puzzled by this, thinking that our proof works in the more complicated case where $n$ is present but fails in the seemingly simpler case where $n$ is absent. In particular, this would betray the simplicity philosophy advocated in Section \ref{S:Significance}. However, naturally, what makes a system of equations simpler or more complex is not the number of equations or variables it has but rather the structures it possesses. From the point of view of the techniques we employ, the case $n \neq 0$ is simpler than the $n=0$ case because it has an additional structure we can exploit, namely, it allows us (using equation \eqref{E:Conservation_n}) to algebraically eliminate $\nabla_\mu u^\mu$ in favor of the simple transport term $\dot{n}/n$, thus leading to equations \eqref{E:ode_for_Density_n_Pi} below which, in turn, give the key equation \eqref{indentity_for_TF} which is used in the proof of Proposition \ref{prop:solving_TF}. Without equation \eqref{E:Conservation_n}, such elimination of the ``complicated" divergence term $\nabla_\mu u^\mu$ in terms of the ``friendlier" transport term $\dot{n}/n$ is not possible. In this regard, it is important to emphasize that our proof relies crucially on estimate for transport equations obtained via integration along the characteristics of the vector-field $u^\mu$, and thus recasting things as transport equations is crucial.
It will be easy to pinpoint exactly where our argument fails in the case $n=0$ after presenting the full proof; see Remark \ref{R:Proof_fails}. We note, however, that there are situations in heavy-ion collisions when $n \neq 0$ (in which case our results can be applied), namely, in low-energy collisions of heavy-ions. The literature on this topic is vast and we refer to the review articles \cite{Lovato:2022vgq,Sorensen:2023zkk}.
\end{remark}

\section{Proof of Theorem \ref{T:Main_theorem_blowup_simpler}}
\label{S:Proof}
In this section we prove Theorem \ref{T:Main_theorem_blowup_simpler}. Thus, we assume
to be working under its hypotheses and follow throughout the notation introduced above. In particular, we assume Assumption \ref{A:Assumptions_on_pressure_and_viscosity}.
Although $\Pi$ vanishes outside of a ball $B_{R(t)}$, we sometimes write $\Pi = \bar{\Pi}$ outside of $B_{R(t)}$ in order to indicate
where the assumption $\bar{\Pi} = 0$ is being used. 
We also denote
$\bar{p} : = p(\bar{\Density},\bar{n})$ and $\mathbf{u} := (u^1,u^2,u^3)$.
Notice that $(u^0)^2 = 1 + \vert\mathbf{u}\vert^2$, where  $\vert \cdot \vert$ is the Euclidean
norm.  

Theorem \ref{T:Main_theorem_blowup_simpler} will readily follow from the following result, whose
proof is the main goal of this Section.

\begin{theorem}
\label{T:Main_theorem_blowup}
Consider the M\"uller-Israel-Stewart equations \eqref{E:Conservation_energy} -- \eqref{E:Bulk_equation} 
(where $g$ is the Minkowski metric). Assume that (A1)-(A5) of Assumption \ref{A:Assumptions_on_pressure_and_viscosity} hold.  
Let $R_0 > 0$ and consider smooth functions $\mathring{\Density}, \mathring{n}, \mathring{\Pi}$ such that 
\begin{align}\label{E:positive_energy_initial}
    \mathring{\Density} + p(\mathring{\Density}, \mathring{n}) + \mathring{\Pi} > 0
\end{align}
and $\mathring{\Density} =\bar \Density, \mathring{n} = \bar n, \mathring{\Pi} = 0$ outside of $B_{R_0}$ and let $\mathbf{u}_1$ be a vector field with support in $B_{R_0}$ such that 
\begin{align}\label{conditions_blowup23}
\int_{B_{R_0}} \mathbf{x} \cdot \mathbf{u}_1 \vert\mathbf{u}_1\vert(\mathring{\Density}+p(\mathring{\Density}, \mathring{n})+\mathring{\Pi})~d x > R_0 \frac{(c+1)^2}{2 (c^2+1)}
\int_{B_{R_0}} \vert\mathbf{u}_1\vert^2 (\mathring{\Density}+p(\mathring{\Density}, \mathring{n})+\mathring{\Pi})~d x,
\end{align}
holds, where $c = c_s(\bar \Density, \bar n, 0) < 1$. Assume moreover
\begin{align}
    (\mathring{\Density},\mathring{n},\mathring{\Pi}) \in \mathcal{P}. 
\end{align}
Then there exists a $\sigma_0 > 0$ such that if $(\Density, n, \Pi, \mathbf{u})$ is the admissible $C^1$-solution taking on the initial values $(\mathring{\Density}, \mathring{n}, \mathring{\Pi}, \sigma\mathbf{u}_1)$ for $\sigma > \sigma_0$, the following is true:
\begin{enumerate}
    \item Outside a ball of radius $R(t) = R_0 + c t$  it holds that
    \begin{align}
        \Density(t, x) = \bar \Density, ~n(t, x) = \bar n, ~\Pi(t, x) = 0, ~\mathbf{u}(t, x) = 0.  
    \end{align}
    \item There exists a finite $T_0 > 0$ such that the solution cannot be continued as a admissible solution past $t=T_0$ (see Definition \ref{D:physical_solution_of_MIS}).
\end{enumerate} 
Finally, the set of smooth $\mathbf{u}_1$ satisfying \eqref{conditions_blowup23} is nonempty. 
\end{theorem}

\begin{remark}
\label{R:Open_data}
Theorem \ref{T:Main_theorem_blowup} implies \emph{there exist a nonempty set of initial velocity fields} such that the solution becomes singular in finite time. This set is open, since the defining conditions \eqref{E:positive_energy_initial} and \eqref{conditions_blowup23} are stable under small perturbations. $\mathring{\Density}, \mathring{n}, \mathring{\Pi}$ can be arbitrary, only satisfying the 
admissibility conditions $(\mathring{\Density}, \mathring{n}, \mathring{\Pi})\in\mathcal{P}$.
\end{remark}

The first major part of the proof consists in deriving an a-priori estimate for the viscous pressure $\Pi$. This will be done in the upcoming Proposition \ref{L:estimate_for_Pi}. A key idea will be to introduce a transport equation in the variables $(\Density, n, \Pi)$ which encodes the evolution of $\Pi$ relative to $\Density, n$ and for which an estimate will be derived in Proposition \ref{prop:solving_TF}.

Suppose $(\Density, n, \Pi, u)$ is a smooth solution of \eqref{E:Conservation_energy}--\eqref{E:Bulk_equation}. Then along a flowline defined by
\begin{align}
    \frac{d}{d\tau} X^\mu(\tau; x_0) = u^\mu(X(\tau; x_0)), ~~X(0; x_0)=(0, x_0)
\end{align}
with given $x_0\in \R^3$, the following equations hold for $\Density = \Density(X(\tau;x_0)), n = n(X(\tau;x_0)), \Pi = \Pi(X(\tau;x_0))$
\begin{align}\label{E:ode_for_Density_n_Pi}
    \begin{split}
        &\dot{\Density} = \frac{(\Density+p(\Density, n)+\Pi)\dot{n}}{n}\\
        &\dot{\Pi} = \frac{\zeta(\Density, n) \dot{n}}{\tau_0(\Density, n) n} - \frac{(1+\lambda(\Density, n)\Pi)\Pi}{\tau_0(\Density, n)} 
    \end{split}    
\end{align}
where a dot denotes differentiation with respect to $\tau$. To derive these equations, eliminate $\nabla_\al u^\al$ from \eqref{E:Conservation_energy} and \eqref{E:Bulk_equation} using \eqref{E:Conservation_n}.  Note also
\begin{align}
    \frac{d}{d\tau}(p+\Pi) = (\del_\rho p) \dot{\rho} + (\del_n p) \dot{n} + \dot{\Pi} = c_s^2 \frac{(\Density+p+\Pi)\dot{n}}{n} - \frac{(1+\lambda \Pi)\Pi}{\tau_0},
\end{align}
where we used the equation for $\dot \Pi$ in \eqref{E:ode_for_Density_n_Pi} and \eqref{eq_sound_speed}.

 Before we can address the estimates for $\Pi$, we need the fact that the weak energy condition propagates for admissible solutions, provided it is satisfied at $t= 0$.
\begin{proposition}\label{Prop:weak_energy}
Assume that the inital data satisfy $\mathring{\Density}+p(\mathring{\Density}, \mathring{n})+\mathring{\Pi} \geq 0$ for all $x\in \R^3$. Then 
$\Density+p+\Pi \geq 0$ for all $(t, x)$.
\end{proposition}
\begin{proof}
Along a trajectory, we have for $e = \Density+p+\Pi$ the following equation:
\begin{align}
    \dot{e} = e \left[(1+c_s^2) \frac{\dot{n}}{n} - \frac{\lambda e + 1}{\tau_0} \right] + \frac{(1-\lambda (\Density+p)) (\Density+p)}{\tau_0}
\end{align}
which is seen by using \eqref{E:ode_for_Density_n_Pi}. We therefore have 
\begin{align}
    e(\tau) = E(\tau) \left(e(0) + \int_0^\tau \frac{(1-\lambda (\Density+p))) (\Density+p)}{E(s)\tau_0}~ds \right)
\end{align}
with 
\begin{align}
    E(\tau) = \exp\left(\int_0^\tau (1+c_s^2) \frac{\dot{n}}{n} - \frac{\lambda e + 1}{\tau_0}~ds\right).
\end{align}
By (A1) and (A5), $\Density+p \geq 0$ and $1-\lambda (\Density+p) \geq 0$. It is now immediate that $e(\tau) \geq 0$, since $e(0) \geq 0$. 
\end{proof}
We will bound solutions of \eqref{E:ode_for_Density_n_Pi} by solving a transport equation in the variables $(\Density, n, \Pi)\in \R\times\R^+\times \R$. To shorten notation, define the folllowing differential operator $\mathcal{T}$, acting on smooth functions $F: \R\times \R^+\times \R\rightarrow \R$:
\begin{align}\label{E:Toward_indentity_for_TF}
    (\mathcal{T} F)(\Density, n, \Pi) = n \del_n F + (\Density+p(\Density, n)+\Pi)\del_\Density F + \frac{\zeta}{\tau_0}\del_{\Pi} F.
\end{align}
Using \eqref{E:ode_for_Density_n_Pi}, we have the identity
\begin{align}\label{indentity_for_TF}
    \frac{\dot n}{n} \mathcal{T}F = \frac{d F}{d\tau}  + \frac{(1+\lambda \Pi)\Pi  }{\tau_0}\del_\Pi F
\end{align}
along a fixed flow-line of the smooth solution.

\begin{proposition}\label{prop:solving_TF}
Under the assumptions (A2)--(A4) (see Assumption \ref{A:Assumptions_on_pressure_and_viscosity}), for any given $\eps>0$ the transport equation
\begin{align}\label{eq_transport_F}
    (\mathcal{T} F)(\Density, n, \Pi) = \frac{\zeta(\Density, n)}{\tau_0(\Density, n)}\quad \text{on}~~\R\times\R^+\times \R
\end{align}
has a smooth solution $F = F(\Density, n, \Pi)$ satisfying
\begin{align}\label{eq_bound_on_F}
\begin{split}
    \vert F(\Density, n, \Pi)\vert \leq \eps+\bar A,\\
    \del_\Pi F(\Density, n, \Pi)  < 0.
\end{split}
\end{align}
Here, $\bar A>0$ is a constant satisfying
\begin{align}
\int_{0}^{\infty} \frac{1}{n}\sup_{\Density \in \R}\left\vert\frac{\zeta(\Density, n)}{\tau_0(\Density, n)}\right\vert~dn \leq \bar A,
\end{align}
which exists by (A4).
\end{proposition}

\begin{proof}
The equation \eqref{eq_transport_F} is linear transport equation with smooth coefficients and can be solved by the method of characteristics. Define a mapping $n\mapsto Y = (Y_\rho, n, Y_\Pi)$ by solving
\begin{align}\label{eq_characteristic_curves_F}
    \frac{d}{dn}Y  = \left(\frac{\Density + p(\Density, n) + \Pi}{n}, 1, \frac{\zeta(\Density, n)}{n \tau_0(\Density, n)}\right)_{(\Density, n, \Pi)=Y}.
\end{align}
Let $n\mapsto Y(n; \Density_0, \Pi_0)$ denote the maximal integral curve passing through the point $(\Density_0, n_0, \Pi_0)$ with some fixed $n_0 > 0$. 

Note that the vector field on the RHS of \eqref{eq_characteristic_curves_F} is defined on $\R\times \R^+\times \R$ and due to assumptions (A2) and (A3), the integral curves of \eqref{eq_characteristic_curves_F} exist for all $n\in (0, \infty)$. The transport equation \eqref{eq_transport_F} can be written as 
\begin{align}
    \frac{d}{dn} (F\circ Y(n; \Density_0, \Pi_0)) = \frac{1}{n}\mathcal{T} F \circ Y= \frac{\zeta}{n \tau_0}\circ Y
\end{align}
and hence we define $F$ at $(\Density, n, \Pi)=Y(n; \Density_0, \Pi_0)$ by integrating along the characteristic
\begin{align}\label{eq_def_of_solution_F}
    F(\Density, n, \Pi) = F_0(\Density_0, \Pi_0) + \int_{n_0}^{n} \frac{\zeta}{ z \tau_0}\circ Y~dz,
\end{align}
where $F_0$ is chosen as a smooth function with the properties
\begin{align}\label{eq_F_0}
     \vert F_0(\Density, \Pi)\vert\leq \eps, ~~\del_\Pi F_0(\Density, \Pi) < 0, ~~\del_\Density F_0(\Density, \Pi) = 0
\end{align}
where $\eps > 0$ is arbitrary.
Due to \eqref{eq_def_of_solution_F} we have   
\begin{align}
    \vert F(\Density, n, \Pi)\vert &\leq \vert F_0( \Density_0, \Pi_0)  \vert + \int_{0}^{\infty} \left\vert\frac{\zeta}{z \tau_0}\circ Y\right\vert ~dz\leq  \eps  + \bar A.
\end{align}
Here we have used (A4) and conclude that the first line of \eqref{eq_bound_on_F} follows. Next, observe that by differentiating \eqref{eq_transport_F} we get
\begin{align}\label{eq_derivatives_of_F}
    \begin{split}
         &\mathcal{T}\del_\Density F = - (1+\del_\Density p) \del_\Density F + \del_\Density\left(\frac{\zeta}{\tau_0}\right)(1-\del_\Pi F),\\
         &\mathcal{T}\del_\Pi F = - \del_\Density F.
    \end{split}
\end{align}
Now fix a characteristic curve $n\mapsto Y(n; \Density_0, \Pi_0)$. The equations \eqref{eq_derivatives_of_F} can be written as
\begin{align}\label{eq_derivatives_of_F1}
    \begin{split}
         &\frac{d}{dn}(\del_\Density F\circ Y) = \left(- \frac{(1+\del_\Density p)}{n} \del_\Density F + \del_\Density\left(\frac{\zeta}{\tau_0}\right)\frac{1-\del_\Pi F}{n}\right)\circ Y,\\
         &\frac{d}{dn}(\del_\Pi F\circ Y) = - \frac{\del_\Density F}{n} \circ Y,
    \end{split}
\end{align}
which can be regarded as a linear system for $\del_\Density F\circ Y$ and $\del_\Pi F\circ Y$ with coefficients $- \frac{(1+\del_\Density p\circ Y)}{n\circ Y}$ etc.
Defining the integrating factor
\begin{align}
    E(n) := \exp\left(\int_{n_0}^n \frac{(1+\del_\Density p)\circ Y}{z} ~dz\right),
\end{align}
multiplying the first line of \eqref{eq_derivatives_of_F1} by $E(n)$ and integrating, we get
\begin{align}\label{eq_solution_del_rho_F}
    \del_\Density F\circ Y(n) = E(n)^{-1}\int_{n_0}^{n} E(s) \frac{1-\del_\Pi F\circ Y}{s} \del_\Density \left(\frac{\zeta}{\tau_0}\right)~d s ,
\end{align}
where we also used $\del_\Density F(\Density, n_0, \Pi)=\del_\Density F_0(\Density, \Pi)=0$. Hence by integrating the second equation of \eqref{eq_derivatives_of_F1} and inserting \eqref{eq_solution_del_rho_F}, 
\begin{align}
    \del_\Pi F\circ Y(n) =  \del_\Pi F_0(\Density_0, \Pi_0) + \int_{n_0}^n \frac{E(s)^{-1}}{s}\int_{n_0}^{s}  E(s')\frac{\del_\Pi F\circ Y(s') - 1}{s'} \del_\Density \left(\frac{\zeta}{\tau_0}\right)~d s'~ ds.  
\end{align}
From \eqref{eq_F_0}, we conclude $\del_\Pi F\circ Y(n) < 0$ if $n$ is close to $n_0$. Assume that $\del_\Pi F$ changes its sign along the characteristic and let $n_1 > n_0$ be the first value of $n > n_0$ for which $\del_\Pi F\circ Y(n_1)= 0$. Then, however
\begin{align}
    0 = \del_\Pi F(\Density_0, \Pi_0) +   \int_{n_0}^{n_1} \frac{E(s)^{-1}}{s}\int_{n_0}^{s}  E(s')\frac{\del_\Pi F\circ Y(s') - 1}{s'} \del_\Density \left(\frac{\zeta}{\tau_0}\right)~d s'~ds < 0
\end{align}
because $\del_\Density \left(\frac{\zeta}{\tau_0}\right) \geq 0$ by (A3) and $(\del_\Pi F\circ Y - 1) \leq 0$ in the integral, a contradiction. In the same way we argue that $\del_\Pi F\circ Y < 0$ for $n < n_0$.
\end{proof}

\begin{proposition}\label{L:estimate_for_Pi}
Assume (A1)--(A5).  Then the following a-priori estimate holds for solutions of the system \eqref{E:ode_for_Density_n_Pi}:
\begin{align}\label{est_Pi}
\vert\Pi(\tau)\vert\leq \vert\Pi(0)\vert+ 3 \bar A
\end{align}
with $\bar A$ as in Proposition \ref{prop:solving_TF}.
\end{proposition}
\begin{proof}
As before, denote $\Pi(\tau) := \Pi(X^\mu(\tau; x_0))$ etc. with a fixed $x_0\in \mathbb{R}^3$ and where $\tau \mapsto X^\mu(\tau; x_0)$ is the flow line starting at $(0, x_0)$. In the following calculations, all quantities are evaluated along a fixed flow line. Combining \eqref{E:ode_for_Density_n_Pi} and \eqref{indentity_for_TF}, we have
\begin{align}
 \dot{\Pi} + \tau_0^{-1}(1+\lambda \Pi)\left(1-\del_\Pi F\right)\Pi =\frac{d}{d\tau}F,
\end{align}
where $F$ is the a solution of $\mathcal{T}F = \zeta/\tau_0$ with the properties described in Proposition \ref{prop:solving_TF}. Multiplying with the integrating factor
\begin{align}
    E(\tau) := \exp\left( \int_0^\tau \tau_0^{-1}(1+\lambda \Pi) (1-\del_{\Pi}F)~ds\right)
\end{align}
we get using integration by parts
\begin{align}
    \Pi(\tau) &= E^{-1}(\tau) \Pi(0) + E^{-1}(\tau) \int_0^{\tau} E(s) \frac{d }{ds} F(\Density(s), n(s), \Pi(s))~ds\\
    &= E^{-1}(\tau) \Pi(0) - E^{-1}(\tau) \int_0^{\tau} E(s) \frac{(1+\lambda \Pi)(1-\del_{\Pi}F)}{\tau_0}F~ds \\
    &\qquad + F(\Density(\tau), n(\tau), \Pi(\tau)) - E^{-1}(\tau) F(\Density(0), n(0), \Pi(0)).
\end{align}
Since our solution satisfies the  {weak energy condition by Proposition \ref{Prop:weak_energy}}, we have $\Pi \geq -(\Density+p)$ and thus $\Pi \geq -\frac{1}{\lambda(\Density, n)}$ by (A5), implying $1+\lambda \Pi \geq 0$. Moreover, $(1-\del_\Pi F)\geq 0$ by \eqref{eq_bound_on_F}. 
Observe that
\begin{align}
    \left\vert\int_0^\tau E(s) \frac{(1+\lambda \Pi)(1-\del_{\Pi}F)}{\tau_0}F~ds\right\vert&\leq (\eps+\bar A) \int_0^\tau E(s) \frac{(1+\lambda \Pi)(1-\del_{\Pi}F)}{\tau_0}~ds\\
    &= (\eps+ \bar A) (E(\tau)-E(0)).
\end{align}
We can hence estimate
\begin{align}
    \vert\Pi(\tau)\vert& \leq \vert\Pi(0)\vert + (\eps+\bar A) E^{-1}(\tau) (E(\tau)-E(0))\\
    &\qquad+ \left\vert F(\Density(\tau), n(\tau), \Pi(\tau)) - E^{-1}(\tau) F(\Density(0), n(0), \Pi(0))\right\vert\\
    &\leq \vert\Pi(0)\vert+ 3(\eps+\bar A) 
\end{align}
and letting $\eps\to 0^+$ finishes the proof.
\end{proof}

This finishes the a-priori estimates for $\Pi$. 
In the next part of the proof, we note a virial-type identity (see \cite{Shadi_YanGuo1999}).
\begin{lemma}
\label{L:Lemma_virial}
Let $\bar{T}$ denote the energy-momentum tensor \eqref{E:Energy_momentum}
associated with $(\bar{\Density}, \bar{n},\bar{\Pi}=0,\bar{\mathbf{u}}=0)$. Then the energy
\begin{align}
    E(t) & := \int_{\mathbb{R}^3} (T^{00}(t,x) - \bar{T}^{00}(t,x))\, dx 
\nonumber
\end{align}
is finite, conserved, and 
\begin{align}
     E(t) &= \int_{B_{R(t)}} (\Density+p+\Pi) (u^0)^2 - (p+\Pi + \bar{\Density} ) \, d x = \int_{B_{R_0}}  (\mathring{\Density}+\mathring{p}+\mathring{\Pi}) \vert\mathring{\mathbf{u}}\vert^2 +\mathring{\Density} - \bar{\Density}  \, d x= E(0).
    \nonumber
\end{align}
Moreover, let
\begin{align}
    I(t) := \frac{1}{2}\int_{\mathbb{R}^3} \vert x\vert ^2 (T^{00} - \bar{T}^{00}) \,d x = \frac{1}{2}\int_{\mathbb{R}^3} \vert x\vert ^2 (T^{00} - \bar{\Density}) \,d x,
    \nonumber
\end{align}
which is finite under our assumptions.
Then the following virial equation
\begin{align}
    \frac{d^2}{dt^2} I(t) =  T + \int_{\mathbb{R}^3} 3(p+\Pi - \bar{p} - \bar{\Pi}) \, dx
\end{align}
holds with kinetic energy
\begin{align}
    T := \int_{B_{R(t)}} (\Density+p +\Pi)\vert\mathbf{u}\vert^2 \, dx.
    \nonumber
\end{align}
Additionally, $\dot{I}(t) = Q(t)$, where
\begin{align}\label{D:def_Q}
    Q(t) := \int_{B_{R(t)}}\mathbf{x} \cdot \mathbf{u} \sqrt{1+\vert\mathbf u\vert^2}(\Density+p+\Pi)~d x.
\end{align}
\end{lemma}
\begin{proof}
The claims on $E$ follow from the fact that $T = \bar{T}$ outside of $B_{R(t)}$
and the conservation  $\partial_\mu T^{\mu \nu} = 0$.

Next, we compute the following derivative for all $t$ such that $R(t) < r$, where $r>0$ is large:
\begin{align}
    2\frac{d}{dt} I &= \int_{\mathbb{R}^3} \vert x\vert ^2 \del_0 (T^{00} - \bar{T}^{00})  \, dx
    = \int_{B_{r}} \vert x\vert ^2 \del_{0} T^{00} \, dx=- \int_{B_{r}} \vert x\vert ^2 \del_{k} T^{0k} \, dx
    \nonumber \\
    &=  2 \int_{B_{r}} g_{ik} x^i  T^{0k} \, dx = 2 \int_{B_{r}} x \cdot \mathbf{u} \, u^0 (\Density+p+\Pi) \, dx = 2 Q(t)
    \nonumber
\end{align}
Continuing,
\begin{align}
    \frac{d^2}{dt^2} I &= \int_{B_{r}} g_{ik} x^i \del_0 T^{0k} \, dx = \int_{B_{r}} g_{ik} x^i \del_0 (T^{0k}-\bar{T}^{0k}) \, dx=- \int_{B_r} g_{ik} x^i \del_j (T^{jk} - \bar{T}^{jk}) \, dx
    \nonumber \\
        &= \int_{B_r} g_{ik} \delta^i_j  (T^{jk} - \bar{T}^{jk}) \, dx = \int_{B_r} ( (\Density+ p+\Pi) \vert\mathbf{u}\vert^2 + 3 (p+\Pi - \bar{p} - \bar{\Pi}) ) \, dx,
    \nonumber
\end{align}
which is our desired result.
\end{proof}
Next we note a crucial a-priori bound for the quantity $Q$:

\begin{lemma}
Suppose $(\Density, n, \Pi, \mathbf{u})$ is a smooth, admissible solution. Let $E=E(0)$ be positive. Then $Q$ satisfies the a-priori bounds
\begin{align}
    Q(t)^2\leq R(t)^2 (2(E+b R(t)^3)-T) T \label{bound_for_Q_with_T}
\end{align}
and
\begin{align}
    \vert Q(t)\vert \leq R(t) \left(E + b R(t)^3\right)\label{bound_for_Q_with_E}
\end{align}
where $b = \frac{4 \pi}{3} (\bar \Density + p_1 + \vert \mathring{\Pi}\vert_\infty + 3 \bar A)$
and $T$ is as in Lemma \ref{L:Lemma_virial}.
\end{lemma}
\begin{proof}
For brevity, we write $Q(t) = Q, R(t)=R$ etc. We estimate, using the inequality $a b \leq (\epsilon/2)a^2 + b^2/(2\epsilon)$, where $\epsilon>0$, and the inequalities  {$\Density+p+\Pi\geq 0$ and 
\begin{align}
    p+\Pi \leq\Density +  p_1 + \vert \mathring{\Pi}\vert _\infty + 3 \bar A =: \Density+A
\end{align} (by Propositions \ref{Prop:weak_energy}, \ref{L:estimate_for_Pi} and (A1)):
\begin{align}
    \vert Q\vert  &\leq R \int_{B_R} \vert \mathbf{u}\vert \vert \sqrt{1+\mathbf{u}^2}\vert  (\Density+p+\Pi)~d x\\
    &\leq R \left(\frac{\epsilon}{2} \int_{B_R} \vert \mathbf{u}\vert ^2 (\Density+p+\Pi) ~d x+
    \frac{1}{2\epsilon} \int_{B_R} (1+\vert \mathbf{u}\vert ^2) (\Density+p+\Pi) ~d x \right)\\
    &\leq R \left[\frac{\epsilon}{2} \int_{B_R} \vert \mathbf{u}\vert ^2 (\Density+p+\Pi) ~d x+
    \frac{1}{2\epsilon} \int_{B_R} \vert \mathbf{u}\vert ^2 (\Density+p+\Pi) ~d x  \right.\\ 
    &\qquad \left.+\frac{2}{2\epsilon} \int_{B_R} \Density ~d x + \frac{4 \pi R^3 A}{6 \epsilon} \right]\\
    &= R \left[\frac{1}{2}\epsilon T + \frac{1}{\epsilon}\left(E + \frac{4 \pi}{3} (\bar \Density + \frac{1}{2} A) R^3 - \frac{T}{2} \right) \right]\\
    &\leq R \left[\frac{1}{2}\epsilon T + \frac{1}{\epsilon}\left(E + b R^3 - \frac{T}{2} \right) \right],
\end{align}}
an inequality holding for all $\epsilon > 0$. Denoting the right-hand-side of this inequality by $R f(\epsilon)$, we observe that $f(\epsilon)\to \infty$ as $\epsilon \to \infty$. 
Note first $E + b R^3 - T/2\geq E + 4\pi/3 \bar \Density R^3 - T/2 = T/2 + \int_{B_R} \Density~d x  \geq 0$. Assuming for the moment that this quantity is $>0$, we get that $f(\epsilon)\to \infty$ as $\epsilon \to 0^+$. We can therefore minimize $f(\epsilon)$ by determining the global minimum by setting $0 = f'(\epsilon) = T/2 - (E+ b R^3 - T/2)\epsilon^{-2}$. A computation yields
\begin{align}
    \min_{\epsilon > 0} f(\epsilon) = \sqrt{[2(E+ b R^3)-T] T},
\end{align}
yielding the inequality \eqref{bound_for_Q_with_T}. If $E + b R^3 - T/2 = 0$, we can send $\epsilon \to 0^+$ to obtain $Q = 0$, hence \eqref{bound_for_Q_with_T} holds in this case as well. 

To obtain \eqref{bound_for_Q_with_E}, we set $\epsilon = 1$. 
\end{proof}

\begin{lemma}\label{lemma:blowupQ} Let $E = E(0)$ be positive.
Suppose that $Q: [0, T_1) \to \mathbb{R}$, defined by \eqref{D:def_Q}, satisfies the differential inequality
\begin{align}\label{lemma:blowupQ_eq}
\dot{Q} \geq E+b R^3 - \sqrt{(E+b R^3)^2-\frac{Q^2}{R^2}} - k R^3
\end{align}
with constants $b, k > 0$ and $R(t) = R_0 + c t$ with $c < 1$. Assume that there exists a $\bar R > R_0$ such that the following holds:
\begin{align}\label{lem_Q:blowup_cond1}
\begin{split}   
    A^2 + 2 B - B^2 &> 0, \\
    A+B &< 1,\\
    z_0 := \frac{A(1-B) + \sqrt{A^2 + 2 B - B^2}}{A^2+1} &< 1,
\end{split}
\end{align}
with
\begin{align}
  A= c \left(1+\frac{3 b \bar{R}^3}{E+b \bar{R}^3}\right), ~B = \frac{ k \bar{R}^3}{E+b \bar{R}^3}.
\end{align}
Assume moreover
\begin{align}\label{lem_Q:blowup_cond2}
   \int_{\frac{1}{2}+\frac{z_0}{2}}^1 \frac{dz}{1- \sqrt{1-z^2} - A z  - B} < \log\left(\frac{\bar R}{R_0}\right)
\end{align}
and
\begin{align}\label{lem_Q:blowup_cond3}
    \frac{Q(0)}{R_0 (E+ b R_0^3)} > \frac{z_0}{2}+\frac{1}{2}.  
\end{align}
Then $Q$ is necessarily defined on a finite interval $[0, T_1)$ with $T_1 < (\bar R-R_0)/c < \infty$ and cannot be extended smoothly past beyond that time as a function satisfying \eqref{lemma:blowupQ_eq}.
\end{lemma}
\begin{proof}
Define 
\begin{align}
    z(t) = \frac{Q}{R (E+ b R^3)}.
\end{align}
Using the differential inequality for $Q$, we obtain the following differential inequality for $z$:
\begin{align}
    \dot{z} \geq \frac{1}{R}\left[1- \sqrt{1-z^2} - c z \left(1+\frac{3 b R^3}{E+b R^3}\right) - \frac{k R^3}{E+b R^3}\right].
\end{align}
Noting that $R\mapsto R^3/(E+b R^3)$ is monotone increasing in $R$, this implies
\begin{align}
    \dot{z} \geq \frac{h(z)}{R_0 + c t} .\label{eq_aux65}
\end{align}
on the interval $[0, (\bar{R}-R_0)/c]$, where
\begin{align}
    h(z) = 1- \sqrt{1-z^2} - A z - B.
\end{align}
Note that due to \eqref{lem_Q:blowup_cond1}, $h(1) > 0$, $h(z_0) = 0$ and $h(z) > 0$ for $z_0 < z \leq 1$. Dividing \eqref{eq_aux65} by $h(z)$ and using $\frac{1}{2}(z_0+1) < z(0) \leq z(t)\leq 1$ (see \eqref{bound_for_Q_with_E}), we obtain after integrating
\begin{align}
    \int_{\frac{1}{2}(z_0+1)}^1 \frac{dz}{h(z)}\geq \int_{z(0)}^{z(t)} \frac{dz}{h(z)} \geq \log\left(\frac{R(t)}{R_0}\right).
\end{align}
This implies via \eqref{lem_Q:blowup_cond2}
\begin{align}
    \log\left(\frac{\bar R}{R_0}\right) > \log\left(\frac{R(t)}{R_0}\right),
\end{align}
a contradiction if $Q$ can be smoothly continued past $\bar T_1$. 
\end{proof}

\begin{lemma}\label{lemma:blowupQ1}
Let the assumptions of Lemma \ref{lemma:blowupQ} hold and $\mathring{\Density}, \mathring{n}, \mathring{\Pi}$ be such that 
\begin{align}
    \mathring{\Density} + p+\mathring{\Pi} > 0
\end{align}
and suppose $\mathbf{u}_1$ is a smooth velocity field supported in $B_{R_0}$ such that
\begin{align}\label{blowup_cond_11}
\int_{B_{R_0}} \mathbf{x} \cdot \mathbf{u}_1 \vert \mathbf{u}_1\vert (\Density+p+\Pi)~d x > R_0 \frac{(c+1)^2}{2 (c^2+1)}
\int_{B_{R_0}} \vert \mathbf{u}_1\vert ^2 (\Density+p+\Pi)~d x.
\end{align}
Then for $\mathring{\mathbf{u}}=\sigma\mathbf{u}_1$, \eqref{lem_Q:blowup_cond1}-\eqref{lem_Q:blowup_cond3} are satisfied for sufficiently large $\sigma > 0$.
\end{lemma}
\begin{proof}
First pick a $\mu > 1$ such that
\begin{align}\label{eq_aux88}
    \int_{\frac{(c+1)^2}{2(c^2+1)}}^1 \frac{dz}{1- \sqrt{1-z^2} - c z} < \log \mu
\end{align}
holds. The choice of $\mu$ only depends on $c < 1$. We set $\bar R =\mu R_0$ and take $\mathring{\mathbf{u}} = \sigma \mathbf{u}_1$ and we will show that the conditions \eqref{lem_Q:blowup_cond1}--\eqref{lem_Q:blowup_cond3} are satisfied for large $\sigma>0$.  
To see this, first note that as $\sigma \to \infty$
\begin{align}
    E = \sigma^2 \int_{B_{R_0}} \vert \mathbf{u}_1\vert ^2(\mathring{\Density} + \mathring{p} + \mathring{\Pi})~d x + \int_{B_{R_0}} (\mathring{\Density}-\bar\Density)~d x \to \infty
\end{align}
and hence 
\begin{align}
A &= c \left(1+\frac{3 b \mu^3 R_0^3}{E+b \mu^3 R_0^3}\right)  \to c, \quad
B = \frac{ k\mu^3 R_0^3}{E+b \mu R_0^3} \to 0, \quad
z_0 \to \frac{2 c}{c^2+1}.
\end{align}
This means that \eqref{lem_Q:blowup_cond1} and \eqref{lem_Q:blowup_cond2} are satisfied for large $\sigma$. It remains to verify \eqref{lem_Q:blowup_cond3}:
\begin{align}
    \frac{Q(0)}{R_0 (E+b R_0^3)} &=  \frac{\int_{B_{R_0}}\mathbf{x} \cdot \mathbf{u}_1 \sqrt{\sigma^{-2}+\vert \mathbf{u}_1\vert ^2}(\Density+p+\Pi)~d x}{R_0\left(\int_{B_{R_0}} \vert \mathbf{u}_1\vert ^2 (\Density+p+\Pi)~d x  + \sigma^{-2}\int_{B_{R_0}} \Density~d x\right)} \\
    &\to  \frac{\int_{B_{R_0}}\mathbf{x} \cdot \mathbf{u}_1 \vert \mathbf{u}_1\vert (\Density+p+\Pi)~d x}{R_0 \int_{B_{R_0}} \vert \mathbf{u}_1\vert ^2 (\Density+p+\Pi)~d x}~~~~~\text{as}~~\sigma\to \infty     
\end{align}
whereas 
\begin{align}
    \frac{1}{2}(z_0+1) \to \frac{(c+1)^2}{2(c^2+1)}
\end{align}
and hence \eqref{lem_Q:blowup_cond3} holds for large enough $\sigma>0$ because of \eqref{blowup_cond_11}.
\end{proof}

We are now ready to establish Theorem \ref{T:Main_theorem_blowup}.
\vskip 0.3cm
\noindent \emph{Proof of Theorem \ref{T:Main_theorem_blowup}:}
Suppose that $(\Density, n, \Pi, \mathbf{u})$ is a smooth admissible solution of the M\"uller-Israel-Stewart equations.
 {The statement about the solution outside the ball $B_{R(t)}$ holds by an argument similar to the one in \cite{Shadi_YanGuo1999} (note that the M\"uller-Israel-Stewart equations can be written as a nonlinear symmetric hyperbolic system, see \cite{BemficaDisconziNoronha_IS_bulk}).}
From Lemma \ref{L:Lemma_virial} we have
\begin{align}  
    \dot{Q} &= \ddot{I} = T + \int_{B_{R(t)}}
    3 (p + \Pi - \bar p -\bar \Pi) ) \, dx
    \geq T - 3\int_{B_{R(t)}} (\vert \Pi\vert +p_0)\,d x  - 3  \int_{B_{R(t)}} \bar{p}\, d x
\end{align}
where we have used $p\geq -p_0$ and $\bar \Pi = 0$. Using Proposition \ref{L:estimate_for_Pi} to estimate $\vert \Pi\vert $, we get
\begin{align}\label{eq_aux33}
    \dot{Q} \geq T - \frac{4\pi}{3}(\vert \mathring{\Pi}\vert _\infty + 3 \bar A +p_0+\bar p) R^{3}.
\end{align}
On the other hand, by \eqref{bound_for_Q_with_T}, 
\begin{align}
    \frac{Q^2}{R^2} \leq [2 (E+b R^3) - T] T
\end{align}
and hence 
\begin{align}
    T \geq (E+b R^3) - \sqrt{(E+b R^3)^2 - \frac{Q^2}{R^2}}.
\end{align}
Combining this with \eqref{eq_aux33}, we get
\begin{align}
  \dot{Q} \geq (E+b R^3) - \sqrt{(E+b R^3)^2 - \frac{Q^2}{R^2}} - \frac{4\pi}{3}(\vert \mathring{\Pi}\vert _\infty + 3 \bar A +p_0+\bar p) R^{3}
\end{align}
i.e. a differential inequality of the form studied in Lemma \ref{lemma:blowupQ};
 also, Lemma \ref{lemma:blowupQ1} can be applied because of the assumptions 
of Theorem \ref{T:Main_theorem_blowup}. Thus, applying these Lemmas, we finish the proof of the breakdown statement of our Theorem.

It remains to be shown that there is a nonempty set of smooth $\mathbf{u}_1$ for which \eqref{conditions_blowup23} holds.  First, take $\mathbf{\tilde{u}}_1$ of the form:
\begin{align}
    \mathbf{\tilde{u}}_1 &=  \left\{\begin{array}{ll} x/\vert x\vert  & \vert x\vert \in [R_0 - \ell, R_0]\\
    0 & \vert x\vert \notin [R_0 - \ell, R_0] \end{array}\right.
\end{align}
for $0 < \ell <  R_0$.
A calculation yields
\begin{align}
    &\frac{\int_{B_{R_0}}\mathbf{x} \cdot \mathbf{\tilde{u}}_1 \vert \mathbf{\tilde{u}}_1\vert (\Density+p+\Pi)~d x}{R_0 \int_{B_{R_0}} \vert \mathbf{\tilde{u}}_1\vert ^2 (\Density+p+\Pi)~d x} = 
    \frac{ \int_{R_0-\ell}^{R_0} r \int_{\del B_{r}} (\Density+p+\Pi)~dS~dr}{R_0\int_{R_0-\ell}^{R_0}  \int_{\del B_{r}} (\Density+p+\Pi)~dS~dr}\\
    &\to \frac{R_0 \int_{\del B_{R_0}} (\Density+p+\Pi)~dS}{R_0 \int_{\del B_{R_0}} (\Density+p+\Pi)~dS} = 1
\end{align}
as $\ell\to 0^+$. Since $(c+1)^2/(2(c^2+1)) < 1$, by choosing $\ell>0$ small,
we may satisfy \eqref{conditions_blowup23}. 
Now, $\mathbf{\tilde{u}}_1$ is not yet the field we want because it is not smooth. But we can take 
$\mathbf{u}_1$ to be a smooth perturbation of $\mathbf{\tilde{u}}_1$ that is still compactly supported in
$B_{R_0}$ and satisfies \eqref{conditions_blowup23}.
In particular, this shows that the velocity field does not need to be radially symmetric.
\qed

We are now ready to prove our main results.

\medskip
\noindent \emph{Proof of Theorem \ref{T:Main_theorem_blowup_simpler}:} We consider the initial
data constructed in Theorem \ref{T:Main_theorem_blowup}. The existence of a unique admissible local-in-time solution satisfying 1) and 2) follows from the local well-posedness and causality established in \cite{BemficaDisconziNoronha_IS_bulk}. There the M\"uller-Israel-Stewart equations were written as a first-order symmetric hyperbolic system with coefficient matrices depending smoothly on $(\Density, n, \mathbf{u})$ for $(\Density, n)\in \mathcal{P}$ and $\mathbf{u}\in \R^3$. The requirements for local existence in \cite{BemficaDisconziNoronha_IS_bulk} are satisfied because of (A2) and $\mathring{\Density}+\mathring{p}+\mathring{\Pi} > 0$.
Invoking Theorem \ref{T:Main_theorem_blowup} we conclude that the solution cannot be continued as a $C^1$-admissible solution past a finite $T_0>0$.

To address the remaining part of the statement, use the mentioned fact that the M\"uller-Israel-Stewart equations can be written as a first-order symmetric hyperbolic system. Knowing this, the statement that either the quantity in \eqref{E:C1} blows up or $(\Density, n, \Pi)$ is not contained in any compact subset of $\mathcal{P}$ is a consequence of Proposition 1.5 of \cite{Taylor3}, together with the following bound for solutions with initial data as in Theorem \ref{T:Main_theorem_blowup_simpler}
\begin{align}
    \vert \mathbf{u}(t, \cdot)\vert _\infty \leq R(T_0) \vert \nabla \mathbf{u}(t, \cdot)\vert _\infty.
\end{align}
The derivation of this bound uses $\mathbf{u}=0$ outside of $B_{R(t)}$.
\qed

\medskip
\noindent \emph{Proof of Theorem \ref{T:Theorem_stability}:} This follows at once from the constructions
in the proof of Theorem \ref{T:Main_theorem_blowup} (see Remark \ref{R:Open_data}). \qed

\medskip
\noindent \emph{Proof of Theorem \ref{T:Main_theorem_Euler}:} In Theorem \ref{T:Main_theorem_blowup},
take $\zeta=0$ and $\mathring{\Pi}=0$. By uniqueness, we conclude that the corresponding
solution to the M\"uller-Israel-Stewart equations is also a solution to the relativistic Euler equations. Note that here Proposition \ref{prop:solving_TF} is not relevant and Proposition \ref{L:estimate_for_Pi} is trivially true.
\qed

\begin{remark}\label{R:Proof_fails}
We can now explain where our proof of Theorem \ref{T:Main_theorem_blowup_simpler} fails if $n=0$. See Remark \ref{R:n_zero_case} for more context.

Let us consider the case where $n$ is absent, so that the evolution is given by
    \begin{align}
        &u^\al \nabla_\al \Density + (\Density+p+\Pi) \nabla_\al u^\al = 0, \label{eq0.1} \\
        &(\Density + p+ \Pi) u^\beta \nabla_\beta u_\al + \Proj_\al^\beta \nabla_\beta(p+\Pi) = 0,
        \\
    &\tau_0 u^\al \nabla_\al \Pi + \Pi +\lambda \Pi^2 +\zeta \nabla_\al u^\al = 0,
    \end{align}
    where $p = p(\rho), \zeta = \zeta(\rho), \tau_0 = \tau_0(\rho), \lambda = \lambda(\rho)$. 
    Using \eqref{eq0.1}, the evolution equation for $\Pi$ reads
    \begin{align}
        \dot \Pi = - \frac{(1+\lambda \Pi)\Pi}{\tau_0} + \frac{\zeta \dot{\rho}}{\tau_0 (\rho+p+\Pi)}
        \label{E}
    \end{align}
    Suppose we try similar arguments as in the proof of Theorem \ref{T:Main_theorem_blowup_simpler}. We then need to generalize the a-priori estimate on $\Pi$. This can be done by replacing the operator $\mathcal T$ in \eqref{E:Toward_indentity_for_TF} by the following operator 
    \begin{align}
        \mathcal{T} F = (\rho + p(\rho) + \Pi) \del_\rho F + \frac{\zeta(\rho)}{\tau_0(\rho)} \del_{\Pi} F \	\label{T}
    \end{align}
    acting on a function $F = F(\rho, \Pi)$. As a consequence of \eqref{E} and \eqref{T}, the following identity holds (the analogue of \eqref{indentity_for_TF}):
    \begin{align}
        \frac{d}{d\tau} F + (\del_\Pi F) \frac{(1+\lambda \Pi)\Pi}{\tau_0} = \frac{\dot \rho}{\rho+p+\Pi} \mathcal{T} F \label{I}
    \end{align} along a flowline of the fluid. 
    
    At this stage, we need to decide on a transport equation of the form
    \begin{align*}
        \mathcal{T} F = (RHS)
    \end{align*}
    where the right-hand side is chosen appropriately, in order to produce a suitable a-priori estimate for $F$ and hence for $\Pi$.
  
    Comparing the evolution equation for $\Pi$ with the identity \eqref{I}, we see that a natural choice of (RHS) would be  
    $(RHS) = \zeta/\tau_0$, leading to the equation $\dot{\Pi} + \tau_0^{-1}(1+\lambda \Pi)\left(1-\del_\Pi F\right)\Pi =\frac{d}{d\tau}F$. The completion of the proof of Proposition \ref{L:estimate_for_Pi} would then proceed as above, provided we can construct a function $F(\rho, \Pi)$ that is bounded by a constant. 

    However, the choice of $\mathcal{T} F = \zeta/\tau_0$ leads to 
    \begin{align*}
         \mathcal{T} F = (\rho + p(\rho) + \Pi) \del_\rho F + \frac{\zeta(\rho)}{\tau_0(\rho)} \del_{\Pi} F = \frac{\zeta(\rho)}{\tau_0(\rho)}
    \end{align*}
    which is equivalent to 
    \begin{align*}
        (\rho + p(\rho) + \Pi) \frac{\tau_0(\rho)}{\zeta_0(\rho)} \del_\rho F + \del_{\Pi} F = 1
    \end{align*}
    which does not lead to a uniform bound for $F$.
\end{remark}

\section{Riemann invariants}
\label{S:Riemann}
We consider the fluid equations in $1+1$-dimensional Minkowski spacetime:
\begin{align}
\label{E:1_1_MIS_system}
\begin{split}
    u^\mu \nabla_\mu \Density + (\Density+q) \nabla_\mu u^\mu = 0,\\
    (\Density+q) u^\mu \nabla_\mu u^\nu + \Proj^{\nu\mu}\nabla_\mu q = 0,\\
    u^\mu \nabla_\mu q + c^2 (\Density+q) \nabla_\mu u^\mu + f = 0 
\end{split}
\end{align}
where we have written $q = \Pi+p, f = \tau_0^{-1}(\Pi + \lambda \Pi^2)$ and 
\begin{align}
    c^2 := \del_{\Density} p + \frac{\zeta}{\tau_0 (\Density+q)}.
    \label{E:Sound_speed_q}
\end{align}
Note also that $-1 = u_\mu u^\mu = -(u^0)^2 + (u^1)^2$. In this Section we
assume that $p$, $\zeta$, $\tau_0$ depend only on $\Density$, i.e. $p = p(\Density)$. 
We can then take $\Density$, $u$, and $q$ as primary variables.

In $1+1$ spacetime, a standard approach to prove the existence of shocks is to diagonalize the principal part of the system by using Riemann invariants. We will show that Riemann invariants do not exist for \eqref{E:1_1_MIS_system}, except possibly when a very special relation holds.

After dividing the second equation of \eqref{E:1_1_MIS_system} by $(u^0)^2$, 
equations \eqref{E:1_1_MIS_system} can be written in the form
\begin{align}\label{E:1_1_MIS_system_quasilinear}
    \mathcal{A}^0 \del_0 \Psi + \mathcal{A}^1 \del_1 \Psi + \mathcal{B} = 0
\end{align}
where $\Psi = \begin{pmatrix} \Density & u^1 & q\end{pmatrix}^T$ and
\begin{align}
    &\mathcal{A}^0 = \begin{pmatrix}
    u^0 & \frac{(\Density+q)u^1}{u^0} & 0 \\ 
    0 & \frac{\Density+q}{u^0} & \frac{u^1}{u^0} \\ 
    0 & \frac{c^2(\Density+q)u^1}{u^0} & u^0 
    \end{pmatrix},~~~
    \mathcal{A}^1 = \begin{pmatrix}
    u^1 & \Density+q & 0 \\ 
    0 & \frac{(\Density+q)u^1}{(u^0)^2} & 1\\ 
    0 & c^2 (\Density+q) & u^1 
    \end{pmatrix},~~~
    \mathcal{B} = \begin{pmatrix} 0 \\ 0 \\ f
    \end{pmatrix}.
\end{align}
The standard strategy to diagonalize the principal part of \eqref{E:1_1_MIS_system_quasilinear} is to show that
\begin{itemize}
    \item the left eigenvectors $\{l^A(\Psi)\}_{A=1, 2, 3}$ of $(\mathcal{A}^0(\Psi))^{-1}\mathcal{A}^1(\Psi)$ are linearly independent
    \item there exist functions $\Lambda^A=\Lambda^A(\Psi), \alpha^A = \alpha^A(\Psi)$ such that
    \begin{align}\label{E:Riemann_Invariants}
        \Lambda^A(\Psi) l^A(\Psi) = \nabla \alpha^A(\Psi)  
    \end{align}
    where $\nabla$ denotes the gradient with respect to $\Psi$. 
\end{itemize}
We check at once that $(\mathcal{A}^0(\Psi))^{-1}$ exists. The eigenvalues and corresponding
eigenvectors of $(\mathcal{A}^0(\Psi))^{-1}\mathcal{A}^1(\Psi)$  are found to be, respectively,
\begin{align}
\begin{split}
\lambda^1 & = \frac{u^1}{u^0}, \, \lambda^2 = \frac{u^1 + c u^0}{c u^1 + u^0},\,
\, \lambda^3 = \frac{-u^1+c u^0}{c u^1 - u^0}, \\
 l^1 & =  \begin{pmatrix} -c^2 & 0 & 1\end{pmatrix}^T,  \\
 l^2 & =  \begin{pmatrix} 0& \frac{(\Density+q)c}{u^0} & 1\end{pmatrix}^T,  \\
  l^3 & =  \begin{pmatrix} 0& -\frac{(\Density+q)c}{u^0} & 1\end{pmatrix}^T.
\end{split}
\nonumber
\end{align}

\begin{theorem}\label{T:Riemann_invariants}
A necessary condition for the existence of Riemann invariants for
the system \eqref{E:1_1_MIS_system} is that $p(\Density) = constant$.
\end{theorem}
\begin{remark}
The condition $p(\Density) = constant$ will not hold for most physical
systems under natural assumptions. Thus, we can say that in general Riemann invariants do not
exist for the system \eqref{E:1_1_MIS_system}.
\end{remark}
\begin{proof}
We will show that the existence of Riemann invariants implies
\begin{align}
\frac{1}{2} \frac{\zeta}{\tau_0(\Density + q)} + \frac{\partial p}{\partial \Density} = 0.
\label{E:Necessary_Riemann}
\end{align}
Multiplying \eqref{E:Necessary_Riemann} by $p+q$ and differentiating with respect to $q$
gives $\partial_\Density p = 0$.

Consider $l^2$, which we abbreviate as 
$l =  l^2  =  \begin{pmatrix} 0& h & 1\end{pmatrix}^T$. We also abbreviate $\Lambda^2 = \Lambda$.
 According to \eqref{E:Riemann_Invariants}, we have
$\operatorname{curl} (\Lambda l ) =0$, which gives:
\begin{align}
\begin{split}
\frac{\partial \Lambda}{\partial \Psi^2} - h \frac{\partial \Lambda}{\partial \Psi^3}
-\Lambda \frac{\partial h}{\partial \Psi^3}  &= 0,\\
-\frac{\partial \Lambda}{\partial \Psi^1} &=0,\\
\frac{\partial \Lambda}{\partial \Psi^1} h + \Lambda \frac{\partial h}{\partial \Psi^3} &=0.
\end{split}
\label{E:curl}
\end{align}
The second equation in \eqref{E:curl} implies that $\Lambda$ is independent of $\Psi^1$. The third
equation then implies that $h$ is independent of $\Psi^3$. Computing 
$\partial_{\Psi^3} h = \partial_{q} h$ and setting
it equal to zero implies
\begin{align}
\frac{\partial c}{\partial q} = -\frac{c}{\Density + q}.
\nonumber
\end{align}
Computing $\partial_q c$ from \eqref{E:Sound_speed_q} and setting both expressions equal to each
other gives \eqref{E:Necessary_Riemann}. 
\end{proof}

\section{Acknowledgements}
The authors would like to thank Jorge Noronha for useful discussions on a preliminary version of this manuscript. MMD gratefully acknowledges support from NSF grant \# 2107701, from a Sloan Research Fellowship provided by the Alfred P. Sloan foundation,
from a Discovery grant administered by Vanderbilt University, and from a Deans' Faculty
Fellowship. VH's work on this project was funded (full or in-part) by the University of Texas at San Antonio, Office of the Vice President for Research, Economic Development, and Knowledge Enterprise.  VH gratefully acknowledges partial support by NSF grants DMS-1614797 and DMS-1810687.

\section{Data Availability statement}
Data sharing not applicable to this article as no datasets were generated or analysed during the current study.

\section{Conflict of Interest Statement}
On behalf of all authors, the corresponding author states that there is no conflict of interest.

\printbibliography

\Addresses

\end{document}